\documentclass[10pt,a4paper]{article}
\usepackage{amsfonts}

\usepackage{mathrsfs}
\usepackage{xcolor}
\usepackage[pagewise]{lineno}

\usepackage{stmaryrd}
\usepackage{amsmath}
\usepackage{amssymb}

\usepackage{bbm}
\usepackage{mathrsfs}
\usepackage{graphicx}
\usepackage{lipsum}

\usepackage{enumerate}
\usepackage{theorem}
\usepackage{indentfirst}
\makeatletter
\def\tank#1{\protected@xdef\@thanks{\@thanks
        \protect\footnotetext[0]{#1}}}
\def\bigfoot{

    \@footnotetext}
\makeatother

\newcommand{\ea}{\end{array}}

\newtheorem{theorem}{Theorem}[section]
\newtheorem{hypothesis}{Hypothesis}[section]

\newtheorem{lemma}{Lemma}[section]
\newtheorem{definition}{Definition}[section]
\newtheorem{Rem}{Remark}[section]

{\theorembodyfont{\rmfamily}
\newtheorem{Exm}{Example}[section]}

\newenvironment{proof}{Proof.}

\allowdisplaybreaks 
\begin{document}
\title{{\Large \bf Distribution Dependent Stochastic Porous Media Equations 
}\footnote{ W. Liu is supported by NSFC (No.~12171208, 11831014, 12090011); W. Hong is supported by NSFC (No.~12171354); J. Gao is supported by
 Postgraduate Research \& Practice Innovation Program of Jiangsu Province (No. KYCX21$\_$2563). }}
\author{{Jingyue Gao$^a$},   {Wei Hong$^b$},  {Wei Liu$^{a}$}\footnote{Corresponding author: weiliu@jsnu.edu.cn}
\\
  \small $a.$ School of Mathematics and Statistics, Jiangsu Normal University, Xuzhou 221116, China \\
  \small $b.$ Center for Applied Mathematics, Tianjin University, Tianjin 300072, China}
\date{}
\maketitle

\begin{center}
\begin{minipage}{155mm}
{\bf Abstract}  Using the generalized variational framework, the strong/weak existence and uniqueness of solutions are derived for a class of distribution dependent stochastic porous media equations on general measure spaces, which also extends the classical well-posedness result of  quasilinear SPDE to the distribution dependent case.

\vspace{1mm} {\bf Keywords:} SPDE; Porous media equation; Distribution dependent; Variational approach

\vspace{1mm}
\end{minipage}
\end{center}

\section{Introduction}
The porous media equation arises originally as a model for gas flow in a porous medium, which is given by the following form
\begin{equation}\label{e0}
\partial_t X(t)=\Delta\Psi(X(t)),
\end{equation}
where $\Psi$ satisfies certain assumptions and a typical example is $\Psi(x)=x^m:=|x|^{m-1}x$ with some constant $m>1$. The solution of Eq.~(\ref{e0}) stands for the density of gas. We refer to \cite{A,P,V07} and references therein for more background of this model.

This work is mainly concerned with the distribution dependent stochastic porous media type equations, in comparison to the deterministic and stochastic models, we consider random force instead of deterministic ones and the coefficients of such equations not only depend on the spatial and time variables, but also on the distribution of solution, which could model the random effects of some macro environment. The stochastic porous media equations (SPMEs) have attracted considerable attentions in the last decades. For instance, the existence and uniqueness of strong solutions for SPMEs were first investigated in \cite{DRRW} in the additive noise case, where $\Psi$ is a continuous function satisfying some monotonicity and growth conditions. Subsequently, this work was extended to  more general case in \cite{RRW} that the monotone nonlinearities $\Psi$ is a $\Delta_2$-regular Young function such that $r\Psi(r)\to \infty$ as $r\to\infty$, which covers also the fast diffusion equations.

Recently, the authors in \cite{BRR} investigated SPMEs on $\mathbb{R}^d$ under two types of conditions with different methods, namely, the Lipschitz condition via variational method and the polynomial growth condition via Yosida approximation method, and later on R\"{o}ckner et al.~\cite{RWX} extended this work to the case of  general measure spaces.~We  refer to \cite{BDR3,BDR1,BDR2,LRS18,LRS21,RRW,RW1} and references therein for further existence and uniqueness results. Moreover, there are also fruitful results in the literature concerning various properties of solutions to SPMEs such as \cite{BGLR,G,GLS} for global random attractors, \cite{L1,RWW,WZ,XZ} for large deviation principle, and \cite{L4,W2} for the Harnack type inequalities and ergodicity.

 The main purpose of this paper is to show the strong/weak existence and uniqueness of solutions to distribution dependent stochastic porous media equations (DDSPMEs), which have the following form
\begin{equation}\label{e1}
dX(t)=L\Psi(t,X(t),\mathcal{L}_{X(t)})dt+B(t,X(t),\mathcal{L}_{X(t)})dW(t),~t\in[0,T],
\end{equation}
where $L$ is a negative definite self-adjoint linear operator, $\mathcal{L}_{X(t)}$ stands for the distribution of $X(t)$, $\{W(t)\}_{t\in [0,T]}$ is a cylindrical Wiener processes defined on a complete filtered probability space $\left(\Omega,\mathcal{F},\mathcal{F}_{t\geq0},\mathbb{P}\right)$ taking values in a separable Hilbert space $(U,\langle\cdot,\cdot\rangle_U)$, the coefficients $\Psi$ and $B$ satisfy some conditions which will be specified later.

Recently, the distribution dependent stochastic (partial) differential equations (DDSDEs/DDSPDEs), also called McKean-Vlasov S(P)DEs or mean-field S(P)DEs, have attracted more and more attentions, we refer to the recent survey article \cite{HRW} for more information of this topic.
One motivation for DDS(P)DEs is due to its wide applications since the evolution of many stochastic systems might rely on both the microcosmic position and the macrocosmic distribution of the particle, another one is due to  their intrinsic link with nonlinear Fokker-Planck-Kolmogorov equation for probability measures (cf. \cite{BR1,BR3,HRW}). More specifically, the distribution density (denoted by $\rho_t$) of solution to a distribution dependent SDE  solves the following nonlinear PDE
\begin{equation*}
\partial_t\rho_t=L^*\rho_t,~~t\geq 0,
\end{equation*}
where $L$ is a second order nonlinear operator and $L^*$ denotes its adjoint operator.
Such type of model can also be described as the weak limit of $N$-interacting particle systems formed by $N$ equations perturbed by independent Wiener processes, while $N$ goes to infinity (the \emph{propagation of chaos}), the reader might see \cite{M1} for more details on this subject.  For more recent results on DDS(P)DEs, one can see \cite{BR2,BLPR,HS,HW,L2,RW3,RW2,RZ21,W1} and references therein. To the best of our knowledge, the references we mentioned above mainly focus on DDSDEs or semilinear DDSPDEs, there are very few results (see \cite{H}) in the literature concerning quasilinear DDSPDEs due to the technical difficulties caused by the nonlinear terms.

Therefore, we aim to investigate a class of distribution dependent quasilinear SPDEs in this work. Fist we want to mention some background and motivation for investigating such type DDSPDEs. Recently, Shen et al. \cite{SSZZ} considered the large $N$ limits of a coupled system of $N$ interacting $\Phi^4$ models on the $2$-dimensional torus $\mathbb{T}^2$, they studied the $N\to\infty$ limit of each component  and proved that a suitable distribution dependent SPDE governs the limiting dynamics. This large $N$ limit problem can  be viewed as a typical mean field limit result in the context of SPDE systems. In addition, the study of mean field limit for SPDE systems also has precursors, see for example \cite[Chapter 9]{KX}. Therefore, the consideration of a class of distribution dependent quasilinear SPDEs is also quite natural.

 It seems to us  that Kaneko \cite{K2} first developed certain $L_p$-Bessel potential spaces corresponding to a sub-Markovian semigroup in order to solve some problems occurred in Dirichlet spaces. A systematic theory on the Bessel potential spaces is developed by Kazumi and Shigekawa \cite{KS}, one can see also \cite{FJS} and references therein for more results. In the current paper, we will investigate a class of DDSPMEs on the dual space of a Dirichlet space.  We first extend the variational framework, which has been established by Pardoux, Krylov and Rozovskii (see e.g. \cite{KR,LR1}) for classical SPDEs, to the distribution dependent case.
It should be pointed out that the author in \cite{H} studied the strong solutions to DDSDEs in finite as well as infinite dimensional cases with delay. More precisely, he applied a fixed-point argument to prove the existence and uniqueness of solutions in  finite dimensional case and then utilized the Galerkin approximating argument  to prove the existence and uniqueness of solutions in infinite dimensional case under the variational framework with coercivity parameter $p\geq2$ (see (H2) in Section 3.1 of \cite{H}).
In this paper, instead of the Galerkin approximating argument employed in \cite{H}, we directly use the fixed-point approach to prove the existence and uniqueness of strong solutions under the variational framework, in particular, we could extend the coercivity parameter to more general case $p>1$, which now coincides with the classical SPDE case (cf.~\cite{LR1}).

Relying on this variational setting, the strong/weak existence and uniqueness of solutions are derived for a class of DDSPMEs on general measure space as the second main result of this paper. We mainly follow the strategy of the work \cite{RWX}  but some extra efforts to estimate the distribution dependent terms
have to be made.
In addition, we now work in the dual space of a Dirichlet space as in \cite{RWX}, which is more general than the classical Sobolev space of order one in $L^2$, so that we could deal with a general negative definite self-adjoint operator $L$. In particular, as mentioned in \cite{RWX}, it is applicable to the fractional Laplace operator,~i.e.~$L=-(-\Delta)^{\alpha}$, $\alpha\in(0,1]$ and the generalized Schr\"{o}dinger operators $L=\Delta+2\frac{\nabla\rho}{\rho}\cdot\nabla$.

The remainder of this manuscript is organized as follows. In section 2, we construct the variational framework for a class of distribution dependent monotone SPDEs and derive the existence and uniqueness of solutions for such models. In section 3, we devote to proving our strong and weak well-posedness result of DDSPMEs on general measure spaces.

\section{Distribution Dependent Monotone SPDEs}
\setcounter{equation}{0}
 \setcounter{definition}{0}
In this section we aim to extend the classical variational framework to  the distribution dependent case, which will be used later to establish our main result in Section 3.

Let $(U,\langle\cdot,\cdot\rangle_U)$ and $(H, \langle\cdot,\cdot\rangle_H) $ be the separable Hilbert spaces, and $H^*$ the dual space of $H$. Let $V$ denote a reflexive Banach space such that the embedding $V\subset H$ is continuous and dense. Identifying $H$ with its dual space by the Riesz isomorphism, we have the following Gelfand triple
$$V\subset H(\cong H^*)\subset V^*.$$
The dualization between $V$ and $V^*$ is denoted by $_{V^*}\langle\cdot,\cdot\rangle_V$. Moreover, it is easy to see that $_{V^*}\langle\cdot,\cdot\rangle_V|_{H\times V}=\langle\cdot,\cdot\rangle_H$. Let $L_2(U,H)$ be the space of all Hilbert-Schmidt operators from $U$ to $H$.

Let $\mathcal{P}(H)$ be the space of all probability measures on $H$ equipped with weak topology and
$$\mathcal{P}_2(H):=\Big\{\mu\in\mathcal{P}(H):\mu(\|\cdot\|_{H}^2):=\int_H\|\xi\|_H^2\mu(d\xi)<\infty\Big\}.$$
Then $\mathcal{P}_2(H)$ is a Polish space under the so-called $L^2$-Wasserstein distance
$$\mathbb{W}_{2,H}(\mu,\nu):=\inf_{\pi\in\mathcal{C}(\mu,\nu)}\Big(\int_{H\times H}\|\xi-\eta\|_H^2\pi(d\xi,d\eta)\Big)^{\frac{1}{2}},~\mu,\nu\in\mathcal{P}_2(H),$$
here $\mathcal{C}(\mu,\nu)$ denote the set of all couplings for $\mu$ and $\nu$, i.e., $\pi\in\mathcal{C}(\mu,\nu)$ is a probability measure on $H\times H$ such that $\pi(\cdot\times H)=\mu$ and $\pi(H\times \cdot)=\nu$. For any $0\leq s<t<\infty$, let $C([s,t];\mathcal{P}_2(H))$ denote the set of all continuous maps from $[s,t]$ to $\mathcal{P}_2(H)$ under $\mathbb{W}_{2,H}$.

Let $T>0$ be fixed. For the progressively measurable maps
$$
A:[0,T]\times\Omega\times V\times\mathcal{P}(H)\rightarrow V^*,~~B:[0,T]\times\Omega\times V\times\mathcal{P}(H)\rightarrow L_2(U,H),
$$
we consider the following distribution dependent SPDE on $H$
\begin{equation}\label{e4}
dX(t)=A(t,X(t),\mathcal{L}_{X(t)})dt+B(t,X(t),\mathcal{L}_{X(t)})dW(t),~X(0)=X_0,
\end{equation}
where $\{W(t)\}_{t\in [0,T]}$ is an $U$-valued cylindrical Wiener process defined on a complete filtered probability space $\left(\Omega,\mathcal{F},\mathcal{F}_{t\geq0},\mathbb{P}\right)$ admits the following representation
$$W(t)=\sum_{k=1}^{\infty}\beta_k(t)e_k,$$
where $\beta_k(t),k\geq1$ are independent standard Brownian motions. Below we write $A(t,u,\mu)$ to denote the map $\omega\longmapsto A(t,\omega,u,\mu)$, similarly for $B(t,u,\mu)$. We suppose that $A$ and $B$ satisfy the following assumptions.
\begin{hypothesis}\label{h1}
There are some constants $\alpha>1$, $c\in\mathbb{R},\delta>0$ and an $(\mathcal{F}_t)$-adapted process $f_{\cdot}\in L^1([0,T]\times\Omega,dt\times\mathbb{P})$ such that the following conditions hold.
\begin{enumerate}
\item [$({\mathbf{H}}{\mathbf{1}})$] $($Demicontinuity$)$ For all $(t,\omega)\in[0,T]\times\Omega$, the map
\begin{eqnarray*}
V\times\mathcal{P}_2(H)\ni(u,\mu)\mapsto_{V^*}\langle A(t,u,\mu),v\rangle_V
\end{eqnarray*}
is continuous.
\item [$({\mathbf{H}}{\mathbf{2}})$] $($Coercivity$)$ For all $u\in V$, $\mu\in\mathcal{P}_2(H)$ and $t\in[0,T]$,
\begin{eqnarray*}
2_{V^*}\langle A(t,u,\mu),u\rangle_V+\|B(t,u,\mu)\|_{L_2(U,H)}^2\leq c\|u\|_H^2+c\mu(\|\cdot\|_H^2)-\delta\|u\|_V^\alpha+f_t~\text{on}~\Omega.
\end{eqnarray*}
\item [$({\mathbf{H}}{\mathbf{3}})$] $($Monotonicity$)$ For all $u,v\in V$ and $\mu,\nu\in\mathcal{P}_2(H)$,
\begin{eqnarray*}
\!\!\!\!\!\!\!\!&&2_{V^*}\langle A(\cdot,u,\mu)-A(\cdot,v,\nu),u-v\rangle_V+\|B(\cdot,u,\mu)-B(\cdot,v,\nu)\|_{L_2(U,H)}^2
\nonumber\\
\!\!\!\!\!\!\!\!&&~~~\leq c\|u-v\|_H^2+c\mathbb{W}_{2,H}(\mu,\nu)^2~\text{on}~[0,T]\times\Omega.
\end{eqnarray*}
\item [$({\mathbf{H}}{\mathbf{4}})$] $($Growth$)$ For all $u\in V$, $\mu\in\mathcal{P}_2(H)$ and $t\in[0,T]$,
\begin{eqnarray*}
\|A(t,u,\mu)\|_{V^*}^{\frac{\alpha}{\alpha-1}}\leq c\|u\|_V^{\alpha}+c\mu(\|\cdot\|_H^2)+f_t~\text{on}~\Omega ,
\end{eqnarray*}
\begin{eqnarray}\label{b1}
\|B(t,u,\mu)\|_{L_2(U,H)}^2\leq c\|u\|_H^2+c\mu(\|\cdot\|_H^2)+f_t~\text{on}~\Omega.
\end{eqnarray}
\end{enumerate}
\end{hypothesis}

\begin{Rem}
Note that if we choose the Gelfand triple $V=H=V^*=\mathbb{R}^d$, then $A(t,\cdot,\cdot)$ is continuous on $\mathbb{R}^d\times\mathcal{P}_2(\mathbb{R}^d)$.
\end{Rem}

\begin{definition}\label{d1}
We call a continuous $H$-valued $(\mathcal{F}_t)_{t\geq 0}$-adapted process $\{X(t)\}_{t\in[0,T]}$ is a solution of Eq.~(\ref{e4}), if for its $dt\times \mathbb{P}$-equivalent class $\hat{X}$
\begin{eqnarray*}
\hat{X}\in L^\alpha\big([0,T]\times\Omega,dt\times\mathbb{P};V\big)\cap L^2\big([0,T]\times\Omega,dt\times\mathbb{P};H\big),
\end{eqnarray*}
where $\alpha$ is the same as defined in $(\mathbf{H2})$ and $\mathbb{P}$-a.s.
\begin{eqnarray*}
X(t)=X(0)+\int_0^t A(s,\bar{X}(s),\mathcal{L}_{\bar{X}(s)})ds+\int_0^t B(s,\bar{X}(s),\mathcal{L}_{\bar{X}(s)})dW(s),~t\in[0,T],
\end{eqnarray*}
where $\bar{X}$ is an $V$-valued progressively measurable $dt\times\mathbb{P}$-version of $\hat{X}$.
\end{definition}

We now give the main result of this section.
\begin{theorem}\label{t3}
 Suppose that $({\mathbf{H}}{\mathbf{1}})$-$({\mathbf{H}}{\mathbf{4}})$ hold.  Then for $X_0\in L^2(\Omega,\mathbb{P};H)$, Eq.~(\ref{e4}) has a unique solution in the sense of Definition \ref{d1} and  there exists  $C_T>0$ such that
\begin{equation}\label{0}
E\Big[\sup_{t\in[0,T]}\|X(t)\|_H^2\Big]\leq C_T\big(1+\mathbb{E}\|X_0\|_H^2\big).
\end{equation}
\end{theorem}
\begin{Rem}
This result is applicable to various stochastic models such as distribution dependent stochastic p-Laplace equations,  stochastic reaction-diffusion equations and also a class of DDSPMEs.  But it should be noted that the main results obtained in Section 3 can not be covered by this setting, and the DDSPME models in Section 3 is more general. Here  we need to use this result to prove the existence and uniqueness of solutions to certain approximating equations of DDSPMEs (see Theorem \ref{t1}).
\end{Rem}

\begin{proof}
We would like to separate the proof into three steps.

\textbf{Step 1:} Instead of the classical finite-dimensional projection arguments of Galerkin type, here we give a different and more succinct proof for the well-posedness of Eq.~(\ref{e4}).~For any $0\leq s<t\leq T$, $\mu({\cdot})\in C([s,T];\mathcal{P}_2(H))$ and $\psi\in\mathcal{P}_2(H)$, we consider the following reference SPDE for initial value $X_{s,s}^{\psi}$ with $\mathcal{L}_{X_{s,s}^{\psi}}=\psi$,
\begin{equation}\label{e5}
dX_{s,t}^{\psi,\mu}=A^\mu(t,X_{s,t}^{\psi,\mu})dt+B^\mu(t,X_{s,t}^{\psi,\mu})dW(t),~t\in[s,T],
\end{equation}
where $A^\mu(t,x):=A(t,x,\mu(t))$ and $B^\mu(t,x):=B(t,x,\mu(t))$. Note that (\ref{e5}) is a classical SPDE (not distribution dependent),  the conditions ({\textbf{H}}{\textbf{1}})-({\textbf{H}}{\textbf{4}}) imply the conditions in \cite[Theorem 5.1.3]{LR1} which turns out that  the existence and uniqueness of solutions to the reference SPDE (\ref{e5}) hold, and the solution $\{X_{s,t}^{\psi,\mu}\}_{t\in[s,T]}$ in the sense of the Definition \ref{d1} is a continuous $H$-valued $(\mathcal{F}_t)_{t\geq 0}$-adapted process fulfilling  $\mathbb{E}\left[\sup_{t\in[s,T]}\|X_{s,t}^{\psi,\mu}\|_H^2\right]<\infty$. Furthermore, it is easy to prove that $\mathcal{L}_{X_{s,\cdot}^{\psi,\mu}}\in C([s,T];\mathcal{P}_2(H))$. Now we define the  map $\Phi_{s,\cdot}^{\psi}:C([s,T];\mathcal{P}_2(H))\to C([s,T];\mathcal{P}_2(H))$,
\begin{equation}\label{1}
\Phi_{s,t}^{\psi}(\mu):=\mathcal{L}_{X_{s,t}^{\psi,\mu}},~t\in[s,T],~\mu\in C([s,T];\mathcal{P}_2(H)),
\end{equation}
for $\{X_{s,t}^{\psi,\mu}\}_{t\in[s,T]}$ solving Eq.~(\ref{e5}). We mention that $(X^{\psi},\mu)$ is a solution of the DDSPDE (\ref{e4}) with the initial distribution $\psi$ if and only if $X_{s,t}^\psi=X_{s,t}^{\psi,\mu}$ and $\mu(t)=\Phi^{\psi}_{s,t}(\mu)$, $t\in[s,T]$. More precisely, the fixed points of map $\Phi^{\psi}_{s,\cdot}$ are exactly solutions of Eq.~(\ref{e4}). To this end, we will verify the contraction of $\Phi^{\psi}_{s,\cdot}$ with respect to the following complete metric
$$d_t(\mu,\nu):=\sup_{r\in[s,t]}e^{-\lambda r}\mathbb{W}_{2,H}(\mu(r),\nu(r)),$$
in the subspace $M_{t}:=\{\mu\in C([s,t];\mathcal{P}_2(H)):\mu(s)=\psi\}$. Here $\mu,\nu\in C([s,t];\mathcal{P}_2(H))$ for $0\leq s<t\leq T$ and $\lambda$ is a positive constant which will be chosen later.

Let $\mu,\nu\in C([s,t];\mathcal{P}_2(H))$ and $X_{s,s}^{\psi}$ be an $\mathcal{F}_s$-measurable r.v.~with $\mathcal{L}_{X_{s,s}^{\psi}}=\psi$, we consider the following SPDEs
$$dX_{s,t}^{\psi,\mu}=A^\mu(t,X_{s,t}^{\psi,\mu})dt+B^\mu(t,X_{s,t}^{\psi,\mu})dW(t),~X_{s,s}^{\psi,\mu}=X_{s,s}^{\psi},~t\in[s,T],$$
$$dX_{s,t}^{\psi,\nu}=A^\nu(t,X_{s,t}^{\psi,\nu})dt+B^\nu(t,X_{s,t}^{\psi,\nu})dW(t),~X_{s,s}^{\psi,\nu}=X_{s,s}^{\psi},~t\in[s,T].$$
Using It\^{o}'s formula for $\|\cdot\|_H^2$ (cf.~\cite[Theorem 4.2.5]{LR1}),
\begin{eqnarray*}
\!\!\!\!\!\!\!\!&&\|X_{s,t}^{\psi,\mu}-X_{s,t}^{\psi,\nu}\|_H^2
\nonumber\\
=~~\!\!\!\!\!\!\!\!&&\int_s^t\Big[2{}_{V^*}\langle A(s,X_{s,r}^{\psi,\mu},\mu(r))-A(s,X_{s,r}^{\psi,\nu},\nu(r)),X_{s,r}^{\psi,\mu}-X_{s,r}^{\psi,\nu}\rangle_{V}
\nonumber\\
\!\!\!\!\!\!\!\!&&+\|B(s,X_{s,r}^{\psi,\mu},\mu(r))-B(s,X_{s,r}^{\psi,\nu},\nu(r))\|_{L_2(U,H)}^2\Big]dr
\nonumber\\
\!\!\!\!\!\!\!\!&&+2\int_s^t\big\langle\big(B(s,X_{s,r}^{\psi,\mu},\mu(r))-B(s,X_{s,r}^{\psi,\nu},\nu(r))\big)dW(r),X_{s,r}^{\psi,\mu}-X_{s,r}^{\psi,\nu}\big\rangle_H.
\end{eqnarray*}
Following from  ({\textbf{H}}{\textbf{3}}) and the product rule, we know that for any $\kappa>0$,
\begin{eqnarray}\label{es0}
\!\!\!\!\!\!\!\!&&e^{-\kappa t}\mathbb{E}\|X_{s,t}^{\psi,\mu}-X_{s,t}^{\psi,\nu}\|_H^2
\nonumber\\
=~~\!\!\!\!\!\!\!\!&&\int_s^te^{-\kappa r}d\big(\mathbb{E}\|X_{s,r}^{\psi,\mu}-X_{s,r}^{\psi,\nu}\|_H^2\big)+\int_s^t\mathbb{E}\|X_{s,r}^{\psi,\mu}-X_{s,r}^{\psi,\nu}\|_H^2de^{-\kappa r}
\nonumber\\
\leq~~\!\!\!\!\!\!\!\!&&-\kappa\int_s^te^{-\kappa r}\mathbb{E}\|X_{s,r}^{\psi,\mu}-X_{s,r}^{\psi,\nu}\|_H^2dr
\nonumber\\
\!\!\!\!\!\!\!\!&&
+c\int_s^te^{-\kappa r}\mathbb{E}\Big[\|X_{s,r}^{\psi,\mu}-X_{s,r}^{\psi,\nu}\|_H^2+\mathbb{W}_{2,H}(\mu(r),\nu(r))^2\Big]dr.
\end{eqnarray}
Taking $\kappa=c$, (\ref{es0}) yields that
\begin{eqnarray}\label{37}
e^{-c t}\mathbb{E}\|X_{s,t}^{\psi,\mu}-X_{s,t}^{\psi,\nu}\|_H^2
\leq c\int_s^te^{-c r}\mathbb{W}_{2,H}(\mu(r),\nu(r))^2dr.
\end{eqnarray}
Consequently, recalling the definition of $d_t$ and noting that the joint distribution of $(X^{\psi,\mu}_{s,\cdot},X^{\psi,\nu}_{s,\cdot})$ is a coupling of $(\Phi^{\psi}_{s,\cdot}(\mu),\Phi^{\psi}_{s,\cdot}(\nu))$ (cf.~\cite[Theorem 3.3]{Lacker}), it is easy to deduce that
\begin{eqnarray}\label{38}
d_t(\Phi^{\psi}_{s,\cdot}(\mu),\Phi^{\psi}_{s,\cdot}(\nu))
\leq c^{\frac{1}{2}}(t-s)^{\frac{1}{2}}d_t(\mu,\nu).
\end{eqnarray}
 Taking $t_0\in(0,\frac{1}{c})$ such that $ct_0<1$, then map $\Phi^{\psi}_{s,\cdot}$ is strictly contraction on $M_{(s+t_0)\wedge T}$ under the metric $d_t$ for each $s\in[0,T)$, hence it has a unique fixed point.

\textbf{Step 2:} Let $s=0$ and $\psi=\mathcal{L}_{X_0}$. According to the Banach fixed point theorem, there is a unique $\mu(t)=\Phi^{\psi}_{0,t}(\mu)$ for any $t\in[0,t_0\wedge T]$ which  together with the definition of map $\Phi^{\psi}_{s,t}$ implies that $X^{\psi,\mu}_{0,t}$ is a solution to Eq.~(\ref{e4}) up to time $t_0\wedge T$. On the other hand, if we take $\mu(t):=\mathcal{L}_{X(t)}$ for each solution of Eq.~(\ref{e4}), then it is easy to infer that $\mu(t)$ is a solution to the equation
\begin{equation}\label{e6}
\mu(t)=\Phi^{\psi}_{0,t}(\mu),~t\in[0,t_0\wedge T].
\end{equation}
Therefore, the uniqueness of Eq.~(\ref{e6}) gives the uniqueness of Eq.~(\ref{e4}).

We remark that if $t_0\geq T$ then the proof of well-posedness to Eq.~(\ref{e4}) is finished. If $t_0<T$, since $t_0$ is independent of $X_0$, we take $s=t_0$ and $\psi=\mathcal{L}_{X(t_0)}$, (\ref{38}) implies that Eq.~(\ref{e4}) has a unique solution $\{X(t)\}_{[t_0,2t_0\wedge T]}$ up to the time $2t_0\wedge T$. Repeating the same procedure we conclude the existence and uniqueness of solutions up to time $T$.

\textbf{Step 3:} It remains to prove the estimate (\ref{0}). Let $X(t)$ be the unique solution to Eq.~(\ref{e4}).
By It\^{o}'s formula (see \cite[Theorem 4.2.5]{LR1}),
\begin{eqnarray*}
\|X(t)\|_H^2=~~\!\!\!\!\!\!\!\!&&\|X_0\|_H^2+\int_0^t2_{V^*}\langle A(s,X(s),\mathcal{L}_{X(s)}),X(s)\rangle_Vds
\nonumber\\
\!\!\!\!\!\!\!\!&&+\int_0^t\|B(s,X(s),\mathcal{L}_{X(s)})\|_{L_2(U,H)}^2ds+M(t),~t\in[0,T],
\end{eqnarray*}
where we denote
\begin{eqnarray*}
M(t):=2\int_0^t\langle X(s),B(s,X(s),\mathcal{L}_{X(s)})dW(s)\rangle_H
\end{eqnarray*}
is a continuous local martingale.

Set a stopping time $\tau_R$  by
$$\tau_R:=\inf\Big\{t\in[0,T]:\|X(t)\|_H\geq R\Big\}\wedge T,~R>0.$$
It is easy to see that $\tau_R\uparrow T$ as $R\uparrow\infty$. Then following from ({\textbf{H}}{\textbf{2}}) that
\begin{eqnarray*}
\!\!\!\!\!\!\!\!&&\mathbb{E}\left[\sup_{s\in[0,t\wedge\tau_R]}\|X(s)\|_H^2\right]+\delta\mathbb{E}\int_0^{t\wedge\tau_R}\|X(s)\|_V^{\alpha}ds
\nonumber\\
\leq~~\!\!\!\!\!\!\!\!&&\mathbb{E}\|X_0\|_H^2+c\mathbb{E}\int_0^{t\wedge\tau_R}\Big(\|X(s)\|_H^{2}+\mathcal{L}_{X(s)}(\|\cdot\|_H^2)\Big)ds
\nonumber\\
\!\!\!\!\!\!\!\!&&+\mathbb{E}\left[\sup_{s\in[0,t\wedge\tau_R]}|M(s)|\right]+\mathbb{E}\int_0^t|f_s|ds.
\end{eqnarray*}
Using the Burkholder-Davis-Gundy's inequality and the growth of $B$ due to (\ref{b1}) leads to
\begin{eqnarray*}
\!\!\!\!\!\!\!\!&&\mathbb{E}\left[\sup_{s\in[0,t\wedge\tau_R]}\|X(s)\|_H^2\right]+\delta\mathbb{E}\int_0^{t\wedge\tau_R}\|X(s)\|_V^{\alpha}ds
\nonumber\\
\leq~~\!\!\!\!\!\!\!\!&&\mathbb{E}\|X_0\|_H^2+c\mathbb{E}\int_0^{t\wedge\tau_R}\Big(\|X(s)\|_H^{2}+\mathcal{L}_{X(s)}(\|\cdot\|_H^2)\Big)ds
\nonumber\\
\!\!\!\!\!\!\!\!&&+8\mathbb{E}\Big(\int_0^{t\wedge\tau_R}\|B(s,X(s),\mathcal{L}_{X(s)})\|_{L_2(U,H)}^2\|X(s)\|_H^2ds\Big)^{\frac{1}{2}}+\mathbb{E}\int_0^t|f_s|ds.
\nonumber\\
\leq~~\!\!\!\!\!\!\!\!&&\mathbb{E}\|X_0\|_H^2+\frac{1}{2}\mathbb{E}\left[\sup_{s\in[0,t\wedge\tau_R]}\|X(s)\|_H^2\right]+C\mathbb{E}\int_0^T|f_s|ds
\nonumber\\
\!\!\!\!\!\!\!\!&&
+C\mathbb{E}\int_0^{t\wedge\tau_R}\Big(\|X(s)\|_H^{2}+\mathcal{L}_{X(s)}(\|\cdot\|_H^2)\Big)ds,
\end{eqnarray*}
where the constant $C>0$ may change from line to line.

Rearranging the above inequality leads to
\begin{eqnarray*}
\!\!\!\!\!\!\!\!&&\mathbb{E}\left[\sup_{s\in[0,t\wedge\tau_R]}\|X(s)\|_H^2\right]+2\delta\mathbb{E}\int_0^{t\wedge\tau_R}\|X(s)\|_V^{\alpha}ds
\nonumber\\
\leq~~\!\!\!\!\!\!\!\!&&2\mathbb{E}\|X_0\|_H^2+C\mathbb{E}\int_0^T|f_s|ds+C\mathbb{E}\int_0^{t}\mathbf{1}_{[0,\tau_R]}(s)\Big(\|X(s)\|_H^{2}+\mathcal{L}_{X(s)}(\|\cdot\|_H^2)\Big)ds
\nonumber\\
\leq~~\!\!\!\!\!\!\!\!&&2\mathbb{E}\|X_0\|_H^2+C\mathbb{E}\int_0^T|f_s|ds+C\int_0^{t}\mathbb{E}\|X(s)\|_H^{2}ds.
\end{eqnarray*}
Let $R\to\infty$, by the monotone convergence theorem we have
\begin{eqnarray*}
\!\!\!\!\!\!\!\!&&\mathbb{E}\left[\sup_{s\in[0,t]}\|X(s)\|_H^2\right]+2\delta\mathbb{E}\int_0^{t}\|X(s)\|_V^{\alpha}ds
\nonumber\\
\leq~~\!\!\!\!\!\!\!\!&&C\mathbb{E}\|X_0\|_H^2+C\mathbb{E}\int_0^T|f_s|ds+C\int_0^{t}\mathbb{E}\left[\sup_{r\in[0,s]}\|X(r)\|_H^{2}\right]ds.
\end{eqnarray*}
Hence, the Gronwall's lemma implies that
\begin{eqnarray}\label{31}
\mathbb{E}\left[\sup_{s\in[0,t]}\|X(s)\|_H^2\right]+2\delta\mathbb{E}\int_0^{t}\|X(s)\|_V^{\alpha}ds\leq Ce^{CT}\big(\mathbb{E}\|X_0\|_H^2+\mathbb{E}\int_0^T|f_s|ds\big),
\end{eqnarray}
which completes the proof of Theorem \ref{t3}.
\end{proof}

\section{Distribution Dependent SPMEs}
\setcounter{equation}{0}
 \setcounter{definition}{0}
In this section, we will investigate the strong/weak existence and uniqueness of solutions for a class of DDSPMEs based on the variational setting established in Section 2.
First we provide some necessary preparations for the function spaces and operators.

Let $(M,\mathcal{B}(M),\mu_M)$ be a Lusin space and $L^2(\mu_M)$ be the space of square integrable functions on $M$ equipped with the norm $|f|_{2}:=\Big(\int_M|f|^2d\mu_M\Big)^{\frac{1}{2}}$ and the scalar product $\langle\cdot,\cdot\rangle_2$, respectively. Denote by  $(L,\mathcal{D}(L))$ a negative definite self-adjoint linear operator generating a strongly continuous (or $C_0$-) contraction sub-Markovian semigroup $\{T_t\}_{t\geq0}$ on $L^2(\mu_M)$ ($i.e.$ for $u\in L^2(\mu_M)$, $0\leq u\leq 1$ implies $0\leq T_tu\leq 1$).
\begin{definition}
Let $(E,\|\cdot\|)$ denote a Banach space. The gamma-transform $V_r$ of a sub-Markovian semigroup $\{T_t\}_{t\geq0}$ on $E$ is given by the Bochner integral
$$V_ru=\Gamma(\frac{r}{2})^{-1}\int_0^\infty s^{\frac{r}{2}-1}e^{-s}T_suds,$$
where $u\in E$ and $r>0$.
\end{definition}

Now we consider a separable Hilbert space $(F_{1,2},\|\cdot\|_{F_{1,2}})$ by $F_{1,2}:=V_1(L^2(\mu_M))$ (cf.~\cite{K2}) equipped with the norm $\|u\|_{F_{1,2}}=|f|_2$, where $f\in L^2(\mu)$ and $u=V_1f$. It is well-known that in this case $V_1=(1-L)^{-\frac{1}{2}}$ (see \cite[Theorem 1.5.3]{FJS}), then it follows that $F_{1,2}=\mathcal{D}((1-L)^{\frac{1}{2}})$ and $\|u\|_{F_{1,2}}=|(1-L)^{\frac{1}{2}}u|_2$. And we use $F_{1,2}^*$ to denote the dual space of $F_{1,2}$, and denote by $\|\cdot\|_{F_{1,2}^*}$ the associated norm and $_{F_{1,2}^*}\langle\cdot,\cdot\rangle_{F_{1,2}}$ the dualization between $F_{1,2}^*$ and $F_{1,2}$.
\begin{Rem}
Note that if the potential space $F_{r,p}$ for $r>0,~p>1$ (in particular $F_{1,2}$) is regular enough, for example, the semigroup $\{T_t\}_{t\geq0}$ is induced by the following elliptic partial differential operator of second order with smooth coefficients
$$L=\sum_{i,j=1}^n a_{ij}(x)\frac{\partial^2}{\partial x_i\partial x_j}+\sum_{i=1}^n b_i(x)\frac{\partial}{\partial x_i}+c(x),$$
where $a_{ij}(x)\in C_b^2(\mathbb{R}^n),1\leq i,j\leq n$, $\sum_{i,j=1}^n a_{ij}(x)\xi^i\xi^j\geq\vartheta|\xi|^2$ for some $\vartheta>0$, $b_i(x)\in C_b^1(\mathbb{R}^n),1\leq i\leq n$ and $c(x)$ is a non-positive bounded function, then $F_{r,p}$ coincides with the classical Sobolev space $W^{r,p}(\mathbb{R}^n)$. For more results on the potential theory we refer to \cite{FJS,JS,K2}.
\end{Rem}

In this section, the Gelfand triple with $H:=F_{1,2}^*$ and $V:=L^2(\mu_M)$ will be chosen to prove our main results. Consider the map $(1-L):F_{1,2}\to F_{1,2}^*$ defined by
$$_{F_{1,2}^*}\langle(1-L)u,v\rangle_{F_{1,2}}:=\int_M(1-L)^{\frac{1}{2}}u\cdot(1-L)^{\frac{1}{2}}vd\mu_M,~~u,v\in F_{1,2}.$$
It is easy to see that this map is well-defined, and now we recall some known results proved in \cite{RWX} for the reader's convenience.
\begin{lemma}\label{l1}
$(1-L):F_{1,2}\to F_{1,2}^*$ is an isometric isomorphism mapping such that
$$\langle(1-L)u,(1-L)v\rangle_{F_{1,2}^*}=\langle u,v\rangle_{F_{1,2}},~u,v\in F_{1,2}.$$
Moreover, $(1-L)^{-1}:F_{1,2}^*\to F_{1,2}$ is the Riesz isomorphism; that is, for $u\in F_{1,2}^*$,
$$\langle u,\cdot\rangle_{F_{1,2}^*}=_{F_{1,2}}\langle (1-L)^{-1}u,\cdot\rangle_{F_{1,2}^*}.$$
\end{lemma}

In fact the space $L^2(\mu_M)$  is a subset of $F_{1,2}^*$ and the embedding $L^2(\mu_M)\subset F_{1,2}^*$ is continuous and dense (see \cite[Lemma 1.5.6]{FJS}). Hence, we have the following Gelfand triple
\begin{equation}\label{triple}
V=L^2(\mu_M)\subset H=F_{1,2}^*(\cong F_{1,2})\subset (L^2(\mu_M))^*=V^*.
\end{equation}
\begin{lemma}\label{l2}
The map $(1-L):F_{1,2}\to (L^2(\mu_M))^*$ has (unique) continuous extension
$$(1-L):L^2(\mu_M)\to (L^2(\mu_M))^*,$$
which is linear isometric. Moreover, for  $u,v\in L^2(\mu_M)$,
$$_{(L^2(\mu_M))^*}\langle(1-L)u,v\rangle_{L^2(\mu_M)}=\int_Mu\cdot vd\mu_M.$$
\end{lemma}
\begin{Rem}
In fact, in this case $(1-L):L^2(\mu_M)\to (L^2(\mu_M))^*$ is an isometric isomorphism map. Indeed, for any $T\in(L^2(\mu_M))^*$, there exists $u\in L^2(\mu_M)$ such that for all $v\in L^2(\mu_M)$,
$$_{(L^2(\mu_M))^*}\langle T,v\rangle_{L^2(\mu_M)}=\langle u,v\rangle_2=\lim_{n\to\infty}\langle u_n,v\rangle_2,$$
where $u_n\in F_{1,2}$ such that $\lim_{n\to\infty}u_n=u$ in $L^2(\mu_M)$. Therefore, for $v\in L^2(\mu_M)$,
\begin{eqnarray*}
_{(L^2(\mu_M))^*}\langle T,v\rangle_{L^2(\mu_M)}=~~\!\!\!\!\!\!\!\!&&\lim_{n\to\infty}\langle u_n,v\rangle_2
\nonumber\\
=~~\!\!\!\!\!\!\!\!&&\lim_{n\to\infty}\langle u_n,(1-L)(1-L)^{-1}v\rangle_2
\nonumber\\
=~~\!\!\!\!\!\!\!\!&&\lim_{n\to\infty}{}_{F_{1,2}^*}\langle (1-L)u_n,(1-L)^{-1}v\rangle_{F_{1,2}}
\nonumber\\
=~~\!\!\!\!\!\!\!\!&&\lim_{n\to\infty}\langle (1-L)u_n,v\rangle_{F_{1,2}^*}
\nonumber\\
=~~\!\!\!\!\!\!\!\!&&\lim_{n\to\infty}{}_{(L^2(\mu_M))^*}\langle (1-L)u_n,v\rangle_{L^2(\mu_M)}
\nonumber\\
=~~\!\!\!\!\!\!\!\!&&_{(L^2(\mu_M))^*}\langle (1-L)u,v\rangle_{L^2(\mu_M)},
\end{eqnarray*}
where we used Lemma \ref{l1} in the fourth step, which implies the assertion.
\end{Rem}

Now we state the precise assumptions on the coefficients of Eq.~(\ref{e1}).
\begin{hypothesis}
\label{a1}
There exist some constants $\alpha_0,\alpha_1,\alpha_2,\alpha_3,K_1,K_2>0$ such that the following conditions hold.
\begin{enumerate}
\item [$({\mathbf{A}}{\mathbf{1}})$] Let $\Psi:[0,T]\times\Omega\times\mathbb{R}\times\mathcal{P}(H)\to\mathbb{R}$ be progressively measurable; that is, for any $t\in[0,T]$, this map restricted to $[0,t]\times\Omega\times\mathbb{R}\times\mathcal{P}(H)$ is $\mathcal{B}([0,t])\times\mathcal{F}_t\times\mathcal{B}(\mathbb{R})\times\mathcal{B}(\mathcal{P}(H))$-measurable. For all $s,r\in\mathbb{R}$ and $\mu,\nu\in\mathcal{P}(H)$
$$\big(\Psi(\cdot,s,\mu)-\Psi(\cdot,r,\nu)\big)(s-r)\geq0  ~~\text{on}~~ [0,T]\times\Omega.$$

\item [$({\mathbf{A}}{\mathbf{2}})$] The map $\Psi:[0,T]\times\Omega\times V\times\mathcal{P}(H)\to V$ satisfies the following Lipschitz nonlinearity
$$|\Psi(\cdot,u,\mu)-\Psi(\cdot,v,\nu)|_2\leq \alpha_0\big(|u-v|_2+\mathbb{W}_{2,H}(\mu,\nu)\big)~\text{on}~[0,T]\times\Omega,$$
where $u,v\in V,~\mu,\nu\in\mathcal{P}_2(H)$. And $\Psi(\cdot,0,\delta_0)\equiv0$,
here $\delta_0$ is the Dirac measure at $0$.

\item [$({\mathbf{A}}{\mathbf{3}})$] For all $u,v\in V$ and $\mu,\nu\in\mathcal{P}_2(H)$,
\begin{eqnarray*}
\!\!\!\!\!\!\!\!&&2\int_M\big(\Psi(\cdot,u,\mu)-\Psi(\cdot,v,\nu)\big)(u-v)d\mu_M
\nonumber\\
~~~~\geq~~\!\!\!\!\!\!\!\!&&\alpha_1|\Psi(\cdot,u,\mu)-\Psi(\cdot,v,\nu)|_2^2-\alpha_2\mathbb{W}_{2,H}(\mu,\nu)^2-\alpha_3\|u-v\|_{F_{1,2}^*}^2~\text{on}~[0,T]\times\Omega.
\end{eqnarray*}

\item [$({\mathbf{A}}{\mathbf{4}})$] $B:[0,T]\times\Omega\times V\times\mathcal{P}(H)\to L_2(U,V)$ is  progressively measurable and satisfies
 \begin{equation}\label{es1}
 \|B(\cdot,u,\mu)-B(\cdot,v,\nu)\|_{L_2(U,H)}^2\leq K_1\big(\|u-v\|_{F_{1,2}^*}^2+\mathbb{W}_{2,H}(\mu,\nu)^2\big)~\text{on}~[0,T]\times\Omega,
 \end{equation}
\begin{equation}\label{33}
\|B(\cdot,u,\mu)\|_{L_2(U,V)}^2\leq K_2\big(1+|u|_2^2+\mu(\|\cdot\|_{F_{1,2}^*}^2)\big)~\text{on}~[0,T]\times\Omega,
\end{equation}
where $u,v\in V,~\mu,\nu\in\mathcal{P}_2(H)$. And $\|B(\cdot,0,\delta_0)\|_{L_2(U,H)}$ is bounded on $[0,T]\times\Omega$.
\end{enumerate}
\end{hypothesis}
\begin{Rem}
(i) We remark that the growth condition (\ref{33}) of $B(t,u,\mu)$ on $L_2(U,V)$ is  assumed to guarantee a priori estimate of solutions on $V$ (see Lemma \ref{l3}).

(ii) In the distribution independent case (see~e.g.~\cite[Section 2.1]{BDR2}), if $({\mathbf{A}}{\mathbf{1}})$ holds and $\Psi:\mathbb{R}\to\mathbb{R}$ is a monotonically nondecreasing Lipschitz function, then it is obvious that
$$(\Psi(r)-\Psi(s))(r-s)\geq \mathrm{Lip}(\Psi)^{-1}(\Psi(r)-\Psi(s))^2,~r,s\in\mathbb{R},$$
where $\mathrm{Lip}(\Psi)$ is the Lipschitz constant of $\Psi$, which implies that $({\mathbf{A}}{\mathbf{3}})$ holds.

(iii) There are several important physical models involved such equations satisfying $({\mathbf{A}}{\mathbf{1}})$-$({\mathbf{A}}{\mathbf{3}})$, e.g. the celebrated two-phase Stefan problem forced by Gaussian noise. This model characterizes the situation that the solidification or melting process is forced by a stochastic heat flow, we refer to \cite{BDR2} (see also \cite{EO}) for the physical motivations.  
\end{Rem}

\begin{Exm}
For the reader's convenience, here we give a concrete example for $B$ satisfying  ({\textbf{A}}{\textbf{4}}) to illustrate the dependence on distribution, which is a straightforward generalization to the infinite dimensional case from well-known example in finite dimensions (cf.~e.g.\cite[Example 2.16]{Lacker}). Let $B_0:[0,T]\times\Omega\times H\to L_2(U,V)$ fulfills that
\begin{equation}\label{34}
\|B_0(\cdot,x)-B_0(\cdot,y)\|_{L_2(U,H)}^2\leq C_0\|x-y\|_H^2,~x,y\in H,
\end{equation}
\begin{equation}\label{35}
\|B_0(\cdot,x)\|_{L_2(U,V)}^2\leq C_1(1+\|x\|_H^2),~x,y\in H.
\end{equation}
For $\alpha\in\mathbb{R}$,  we consider the following map
$$B^{\alpha}(t,u,\mu):=\int_HB_0(t,u-\alpha z)\mu(dz), ~~  u\in V, \mu\in\mathcal{P}_2(H). $$
Then ({\textbf{A}}{\textbf{4}}) holds for $B=B^{\alpha}$.

\begin{proof}
For $\alpha\in\mathbb{R}$, $u,v\in V$, $\mu,\nu\in\mathcal{P}_2(H)$  and $\pi\in\mathcal{C}(\mu,\nu)$, we have
\begin{eqnarray*}
\!\!\!\!\!\!\!\!&&\|B^{\alpha}(\cdot,u,\mu)-B^{\alpha}(\cdot,v,\nu)\|_{L_2(U,H)}^2
\nonumber\\
=\!\!\!\!\!\!\!\!&&\Big\|\int_HB_0(\cdot,u-\alpha z)\mu(dz)-\int_HB_0(\cdot,v-\alpha \tilde{z})\nu(d\tilde{z})\Big\|_{L_2(U,H)}^2
\nonumber\\
=\!\!\!\!\!\!\!\!&&\Big\|\int_{H\times H}\Big[\big(B_0(\cdot,u-\alpha z)-B_0(\cdot,v-\alpha z)\big)+\big(B_0(\cdot,v-\alpha z)-B_0(\cdot,v-\alpha \tilde{z})\big)\Big]\pi(dz,d\tilde{z})\Big\|_{L_2(U,H)}^2
\nonumber\\
\leq\!\!\!\!\!\!\!\!&&2C_0\|u-v\|_H^2+2\alpha^2\int_{H\times H}\|z-\tilde{z}\|_H^2\pi(dz,d\tilde{z}),
\end{eqnarray*}
where we used (\ref{34}) in the last step. Then
$$\|B^{\alpha}(\cdot,u,\mu)-B^{\alpha}(\cdot,v,\nu)\|_{L_2(U,H)}^2\leq2C_0\|u-v\|_H^2+2\alpha^2\mathbb{W}_{2,H}(\mu,\nu)^2.$$
Meanwhile, taking (\ref{35}) into account, for any $\alpha\in\mathbb{R}$, $u\in V$, $\mu\in\mathcal{P}_2(H)$, we have
\begin{eqnarray*}
\|B^{\alpha}(\cdot,u,\mu)\|_{L_2(U,V)}^2=\!\!\!\!\!\!\!\!&&\Big\|\int_HB_0(\cdot,u-\alpha z)\mu(dz)\Big\|_{L_2(U,V)}^2
\nonumber\\
\leq\!\!\!\!\!\!\!\!&&C_1\int_H(1+\|u-\alpha z\|_H^2)\mu(dz)
\nonumber\\
\leq\!\!\!\!\!\!\!\!&&C_2\int_H(1+|u|_2^2+\alpha^2\|z\|_H^2)\mu(dz)
\nonumber\\
\leq\!\!\!\!\!\!\!\!&&C_3\big(1+|u|_2^2+\mu(\|\cdot\|_H^2)\big),
\end{eqnarray*}
for some constants $C_2,C_3>0$.

Hence ({\textbf{A}}{\textbf{4}}) holds for $B=B^{\alpha}$, $K_1=2(C_0\vee \alpha^2)$ and $K_2=C_3$.
\end{proof}
\end{Exm}

Below we recall the definitions of strong and weak solutions to Eq.~(\ref{e1}).
\begin{definition}
We call a continuous $(\mathcal{F}_t)_{t\geq0}$-adapted process $X:[0,T]\to F_{1,2}^*$ is a (strong) solution to Eq.~(\ref{e1}) with initial point $X(0)\in L^2(\Omega,\mathbb{P};F_{1,2}^*)$, if for any $T>0$,
$$X\in L^2([0,T]\times\Omega;L^2(\mu_M))\cap L^2(\Omega;C([0,T];F_{1,2}^*)),$$
$$\int_0^{\cdot}\Psi(s,X(s),\mathcal{L}_{X(s)})ds\in C([0,T];F_{1,2}),~\mathbb{P}\text{-a.s.},$$
and the following identity holds $\mathbb{P}\text{-a.s.}$,
$$X(t)-L\int_0^t\Psi(s,X(s),\mathcal{L}_{X(s)})ds=X(0)+\int_0^tB(s,X(s),\mathcal{L}_{X(s)})dW(s),~t\in[0,T].$$
\end{definition}

\begin{definition}
(i) A pair $(\tilde{X}(t),\tilde{W}(t))$ is called a weak solution to Eq.~(\ref{e1}), if there exists a cylindrical Wiener process $\{\tilde{W}(t)\}_{t\geq 0}$ with respect to the stochastic basis $(\tilde{\Omega},\{\tilde{\mathcal{F}}_t\}_{t\geq 0},\tilde{\mathbb{P}})$ such that $(\tilde{X}(t),\tilde{W}(t))$ solves the following DDSPDE:
$$\tilde{X}(t)-L\int_0^t\Psi(s,\tilde{X}(s),\mathcal{L}_{\tilde{X}(s)})ds=\tilde{X}(0)+\int_0^tB(s,\tilde{X}(s),\mathcal{L}_{\tilde{X}(s)})d\tilde{W}(s),~t\in[0,T].$$

(ii) We say Eq.~(\ref{e1}) has weak uniqueness in $\mathcal{P}_2(H)$ if $(\tilde{X}(t),\tilde{W}(t))$ with respect to the stochastic basis $(\tilde{\Omega},\{\tilde{\mathcal{F}}_t\}_{t\geq 0},\tilde{\mathbb{P}})$ and $(\bar{X}(t),\bar{W}(t))$ with respect to $(\bar{\Omega},\{\bar{\mathcal{F}}_t\}_{t\geq 0},\bar{\mathbb{P}})$ are two weak solutions to Eq.~(\ref{e1}), then $\mathcal{L}_{\tilde{X}(0)}|_{\tilde{\mathbb{P}}}=\mathcal{L}_{\bar{X}(0)}|_{\bar{\mathbb{P}}}\in\mathcal{P}_2(H)$ implies that $\mathcal{L}_{\tilde{X}(t)}|_{\tilde{\mathbb{P}}}=\mathcal{L}_{\bar{X}(t)}|_{\bar{\mathbb{P}}}\in\mathcal{P}_2(H)$.
\end{definition}

We now formulate the main existence and uniqueness result of this section.
\begin{theorem}\label{t1}
Suppose that  $({\mathbf{A}}{\mathbf{1}})$-$({\mathbf{A}}{\mathbf{4}})$ hold.
For any initial condition $X_0\in L^2(\Omega,\mathbb{P};V)$,  Eq.~(\ref{e1}) has strong/weak existence and uniqueness of solutions, which fulfills that
$$\mathbb{E}\big[\sup_{t\in[0,T]}|X(t)|_2^2\big] < \infty.$$
\end{theorem}

%

\subsection{Approximations to Eq.~(\ref{e1})}
In order to prove the main result, we consider the following approximating equation
\begin{equation}\label{e2}
\left\{ \begin{aligned}
&dX^\epsilon(t)=(L-\epsilon)\Psi(t,X^\epsilon(t),\mathcal{L}_{X^\epsilon(t)})dt+B(t,X^\epsilon(t),\mathcal{L}_{X^\epsilon(t)})dW(t),\\
&X^\epsilon(0)=X(0),
\end{aligned} \right.
\end{equation}
here $\epsilon\in(0,1)$.

\begin{theorem}\label{t2}
Assume that the conditions $({\mathbf{A}}{\mathbf{1}})$-$({\mathbf{A}}{\mathbf{4}})$ hold.
For any initial point $X(0)\in L^2(\Omega,\mathbb{P};V)$, there exists a unique solution denoted by $\{X^\epsilon(t)\}_{t\geq0}$ to Eq.~(\ref{e2}) such that for any $T>0$,
\begin{equation}\label{4}
X^\epsilon\in L^2([0,T]\times\Omega;V)\cap L^2(\Omega;C([0,T];H),
\end{equation}
and $\mathbb{P}\text{-a.s.}$,
\begin{equation}\label{5}
X^\epsilon(t)+(\epsilon-L)\int_0^t\Psi(s,X^\epsilon(s),\mathcal{L}_{X^\epsilon(s)})ds=X(0)+\int_0^tB(s,X^\epsilon(s),\mathcal{L}_{X^\epsilon(s)})dW(s),~t\in[0,T].
\end{equation}
Moreover, the following estimate fulfills for each $\epsilon\in(0,1)$,
\begin{equation}\label{6}
\mathbb{E}\big[\sup_{t\in[0,T]}|X^\epsilon(t)|_2^2\big]\leq C_T.
\end{equation}
\end{theorem}

\begin{proof} First, it is easy to see that the drift of approximating equation (\ref{e2}) does not satisfy the coercivity condition ({\textbf{H}}{\textbf{2}}) in Hypothesis \ref{h1}. In this case, we
consider the following  approximating equation of Eq.~(\ref{e2}) with an additional control term, i.e., for any $t\in[0,T]$ and $\lambda\in(0,1)$,
\begin{equation}\label{e3}
\left\{ \begin{aligned}
&dX^\epsilon_\lambda(t)=(L-\epsilon)\big(\Psi(t,X^\epsilon_\lambda(t),\mathcal{L}_{X^\epsilon_\lambda(t)})+\lambda X^\epsilon_\lambda(t)\big)dt+B(t,X^\epsilon_\lambda(t),\mathcal{L}_{X^\epsilon_\lambda(t)})dW(t),\\
&X^\epsilon_\lambda(0)=X(0).
\end{aligned} \right.
\end{equation}
Then it is straightforward to prove that Eq.~(\ref{e3}) satisfies the conditions ({\textbf{H}}{\textbf{1}})-({\textbf{H}}{\textbf{4}}) in Hypothesis \ref{h1} with help of the perturbation $\lambda X^\epsilon_\lambda(t)$ (see \cite{RWX} for the similar proof in the case of distribution independence).
Consequently, Theorem \ref{t3} gives that there exists a unique solution to Eq.~(\ref{e3}) fulfilling $X^\epsilon_\lambda\in L^2([0,T]\times\Omega;V)\cap L^2(\Omega;C([0,T];H))$ and $\mathbb{P}$\text{-a.s.},
\begin{equation}\label{11.1}
X^\epsilon_\lambda(t)+\int_0^t(\epsilon-L)\big(\Psi(s,X^\epsilon_\lambda(s),\mathcal{L}_{X^\epsilon_\lambda(s)})+\lambda X^\epsilon_\lambda(s)\big)ds=X(0)+\int_0^tB(s,X^\epsilon_\lambda(s),\mathcal{L}_{X^\epsilon_\lambda(s)})dW(s),
\end{equation}
since operator $1-L$ is a linear isometric from Lemma \ref{l2}, then Eq.~(\ref{11.1}) is equivalent to the following equation, $\mathbb{P}$\text{-a.s.},
\begin{equation}\label{11}
X^\epsilon_\lambda(t)+(\epsilon-L)\int_0^t\big(\Psi(s,X^\epsilon_\lambda(s),\mathcal{L}_{X^\epsilon_\lambda(s)})+\lambda X^\epsilon_\lambda(s)\big)ds=X(0)+\int_0^tB(s,X^\epsilon_\lambda(s),\mathcal{L}_{X^\epsilon_\lambda(s)})dW(s).
\end{equation}
\end{proof}
Blow we will verify that by taking $\lambda\to0$ to Eq.~(\ref{e3}), $X^\epsilon_\lambda$ will converge to a solution of Eq.~(\ref{e2}). To this end, we need the initial value $X(0)\in L^2(\Omega,\mathbb{P};V)$ to get the following uniform estimate.
\begin{lemma}\label{l3}
Under the assumptions of Theorem \ref{t1}, there is a constant $C_T$ only depending on $T$ such that
$$\mathbb{E}\big[\sup_{t\in[0,T]}|X^\epsilon_\lambda(t)|_{2}^2\big]+4\lambda\mathbb{E}\int_0^T\|X^\epsilon_\lambda(t)\|_{F_{1,2}}^2dt\leq C_T.$$
Moreover, $X^\epsilon_\lambda$ has $\mathbb{P}$\text{-a.s.}~continuous paths in $L^2(\mu_M)$.
\end{lemma}
\begin{proof}
For each $\delta>\epsilon$, recalling the map $(\delta-L)^{-\frac{1}{2}}:F_{1,2}^*\to L^2(\mu_M)$ and applying it to Eq.~(\ref{11.1}) yields
\begin{eqnarray*}
(\delta-L)^{-\frac{1}{2}}X^\epsilon_\lambda(t)=\!\!\!\!\!\!\!\!&&(\delta-L)^{-\frac{1}{2}}X(0)+\int_0^t(L-\epsilon)(\delta-L)^{-\frac{1}{2}}\big(\Psi(s,X^\epsilon_\lambda(s),\mu^\epsilon_\lambda(s))+\lambda X^\epsilon_\lambda(s)\big)ds
\nonumber\\
\!\!\!\!\!\!\!\!&&+\int_0^t(\delta-L)^{-\frac{1}{2}}B(s,X^\epsilon_\lambda(s),\mu^\epsilon_\lambda(s))dW(s),
\end{eqnarray*}
here we denote $\mu^\epsilon_\lambda(s):=\mathcal{L}_{X^\epsilon_\lambda(s)}$ for convenience. It is obvious that one can consider this equation now under a new Gelfand triple $F_{1,2}\subset L^2(\mu_M)\subset F_{1,2}^*$ which will guarantee an estimate of $X^\epsilon_\lambda$ on $V$.

According to It\^{o}'s formula (cf.~\cite[Theorem 4.2.5]{LR1}), we have
\begin{eqnarray}\label{13}
\!\!\!\!\!\!\!\!&&|(\delta-L)^{-\frac{1}{2}}X^\epsilon_\lambda(t)|_2^2
\nonumber\\
=\!\!\!\!\!\!\!\!&&|(\delta-L)^{-\frac{1}{2}}X(0)|_2^2+2\int_0^t \!_{F_{1,2}^*}\langle(L-\epsilon)(\delta-L)^{-\frac{1}{2}}\Psi(s,X^\epsilon_\lambda(s),\mu^\epsilon_\lambda(s)),(\delta-L)^{-\frac{1}{2}}X^\epsilon_\lambda(s)\rangle_{F_{1,2}}ds
\nonumber\\
\!\!\!\!\!\!\!\!&&+2\lambda\int_0^t \!_{F_{1,2}^*}\langle(L-\epsilon)(\delta-L)^{-\frac{1}{2}}X^\epsilon_\lambda(s),(\delta-L)^{-\frac{1}{2}}X^\epsilon_\lambda(s)\rangle_{F_{1,2}}ds
\nonumber\\
\!\!\!\!\!\!\!\!&&+\int_0^t\|(\delta-L)^{-\frac{1}{2}}B(s,X^\epsilon_\lambda(s),\mu^\epsilon_\lambda(s))\|_{L_2(U,V)}^2ds
\nonumber\\
\!\!\!\!\!\!\!\!&&+2\int_0^t\langle(\delta-L)^{-\frac{1}{2}}B(s,X^\epsilon_\lambda(s),\mu^\epsilon_\lambda(s))dW(s),(\delta-L)^{-\frac{1}{2}}X^\epsilon_\lambda(s)\rangle_2.
\end{eqnarray}

Letting $P:=(\delta-\epsilon)(\delta-L)^{-1}$. For any $f\in L^2(\mu_M)$, it is easy to obtain that
\begin{eqnarray*}
(P-I)f=\!\!\!\!\!\!\!\!&&[(\delta-L)^{-\frac{1}{2}}(\delta-\epsilon)(\delta-L)^{-\frac{1}{2}}-(\delta-L)^{-\frac{1}{2}}(\delta-L)(\delta-L)^{-\frac{1}{2}}]f
\nonumber\\
=\!\!\!\!\!\!\!\!&&[(\delta-L)^{-\frac{1}{2}}(L-\epsilon)(\delta-L)^{-\frac{1}{2}}]f.
\end{eqnarray*}
We remark that $P$ is obviously a sub-Markovian operator since the semigroup $\{T_t\}_{t\geq 0}$ associated with $L$ is sub-Markovian. For the contraction of $P$ on $L^2(\mu_M)$,
\begin{eqnarray*}
|Pf|_2^2=|(\delta-\epsilon)(\delta-L)^{-1}f|_2^2=\frac{(\delta-\epsilon)^2}{\delta^2}|\delta(\delta-L)^{-1}f|_2^2\leq|f|_2^2 ~\text{for any}~f\in L^2(\mu_M).
\end{eqnarray*}
Consequently, $P$ is a symmetric contraction sub-Markovian operator.~Denote by $p_\delta$ the probability kernel corresponding to $P$.~The first integral on the right hand side of (\ref{13}) is equivalent to
\begin{eqnarray*}
\!\!\!\!\!\!\!\!&&2\int_0^t \!_{F_{1,2}^*}\langle(L-\epsilon)(\delta-L)^{-\frac{1}{2}}\Psi(s,X^\epsilon_\lambda(s),\mu^\epsilon_\lambda(s)),(\delta-L)^{-\frac{1}{2}}X^\epsilon_\lambda(s)\rangle_{F_{1,2}}ds
\nonumber\\
=\!\!\!\!\!\!\!\!&&2\int_0^t-{}_{F_{1,2}^*}\langle(1-L)(\delta-L)^{-\frac{1}{2}}\Psi(s,X^\epsilon_\lambda(s),\mu^\epsilon_\lambda(s)),(\delta-L)^{-\frac{1}{2}}X^\epsilon_\lambda(s)\rangle_{F_{1,2}}
\nonumber\\
\!\!\!\!\!\!\!\!&&+\langle(1-\epsilon)(\delta-L)^{-\frac{1}{2}}\Psi(s,X^\epsilon_\lambda(s),\mu^\epsilon_\lambda(s)),(\delta-L)^{-\frac{1}{2}}X^\epsilon_\lambda(s)\rangle_{2}ds
\nonumber\\
=\!\!\!\!\!\!\!\!&&2\int_0^t-\langle(1-L)^{\frac{1}{2}}(\delta-L)^{-\frac{1}{2}}\Psi(s,X^\epsilon_\lambda(s),\mu^\epsilon_\lambda(s)),(1-L)^{\frac{1}{2}}(\delta-L)^{-\frac{1}{2}}X^\epsilon_\lambda(s)\rangle_{2}
\nonumber\\
\!\!\!\!\!\!\!\!&&+\langle\Psi(s,X^\epsilon_\lambda(s),\mu^\epsilon_\lambda(s)),(\delta-L)^{-\frac{1}{2}}(1-\epsilon)(\delta-L)^{-\frac{1}{2}}X^\epsilon_\lambda(s)\rangle_{2}ds
\nonumber\\
=\!\!\!\!\!\!\!\!&&2\int_0^t\langle\Psi(s,X^\epsilon_\lambda(s),\mu^\epsilon_\lambda(s)),PX^\epsilon_\lambda(s)-X^\epsilon_\lambda(s)\rangle_2ds.
\end{eqnarray*}
Then taking \cite[Lemma 5.1]{RW1} and ({\textbf{A}}{\textbf{1}}) into account, it implies that
\begin{eqnarray*}
\!\!\!\!\!\!\!\!&&2\int_0^t\langle\Psi(s,X^\epsilon_\lambda(s),\mu^\epsilon_\lambda(s)),PX^\epsilon_\lambda(s)-X^\epsilon_\lambda(s)\rangle_2ds
\nonumber\\
=\!\!\!\!\!\!\!\!&&-\int_0^t\Big\{\int_M\int_M\Big[\Psi(s,X^\epsilon_\lambda(s)(\xi),\mu^\epsilon_\lambda(s))-\Psi(s,X^\epsilon_\lambda(s)(\tilde{\xi}),\mu^\epsilon_\lambda(s))\Big]
\Big[X^\epsilon_\lambda(s)(\xi)-X^\epsilon_\lambda(s)(\tilde{\xi})\Big]
\nonumber\\
\!\!\!\!\!\!\!\!&&\cdot p_\delta(\xi,\tilde{\xi})\mu_M(d\tilde{\xi})\mu_M(d\xi)\Big\}ds-2\int_0^t\Big\{\int_M(1-P1)\Big[X^\epsilon_\lambda(s)\Psi(s,X^\epsilon_\lambda(s),\mu^\epsilon_\lambda(s))\Big]d\mu_M\Big\}ds
\nonumber\\
\leq\!\!\!\!\!\!\!\!&&0.
\end{eqnarray*}
Meanwhile the second integral on the right hand side of (\ref{13}),
\begin{eqnarray*}
\!\!\!\!\!\!\!\!&&2\lambda\int_0^t \!_{F_{1,2}^*}\langle(L-\epsilon)(\delta-L)^{-\frac{1}{2}}X^\epsilon_\lambda(s),(\delta-L)^{-\frac{1}{2}}X^\epsilon_\lambda(s)\rangle_{F_{1,2}}ds
\nonumber\\
=\!\!\!\!\!\!\!\!&&-2\lambda\int_0^t {}_{F_{1,2}^*}\langle(1-L)(\delta-L)^{-\frac{1}{2}}X^\epsilon_\lambda(s),(\delta-L)^{-\frac{1}{2}}X^\epsilon_\lambda(s)\rangle_{F_{1,2}}ds
\nonumber\\
\!\!\!\!\!\!\!\!&&+2\lambda\int_0^t \langle(1-\epsilon)(\delta-L)^{-\frac{1}{2}}X^\epsilon_\lambda(s),(\delta-L)^{-\frac{1}{2}}X^\epsilon_\lambda(s)\rangle_{2}ds
\nonumber\\
\leq\!\!\!\!\!\!\!\!&&-2\lambda\int_0^t \|(\delta-L)^{-\frac{1}{2}}X^\epsilon_\lambda(s)\|_{F_{1,2}}^2ds+2\int_0^t |(\delta-L)^{-\frac{1}{2}}X^\epsilon_\lambda(s)|_2^2ds.
\end{eqnarray*}
Hence, multiplying $\delta$ for both sides of (\ref{13}) yields that
\begin{eqnarray*}
\!\!\!\!\!\!\!\!&&|\sqrt{\delta}(\delta-L)^{-\frac{1}{2}}X^\epsilon_\lambda(t)|_2^2+2\lambda\int_0^t \|\sqrt{\delta}(\delta-L)^{-\frac{1}{2}}X^\epsilon_\lambda(s)\|_{F_{1,2}}^2ds
\nonumber\\
\leq\!\!\!\!\!\!\!\!&&|\sqrt{\delta}(\delta-L)^{-\frac{1}{2}}X(0)|_2^2+2\int_0^t |\sqrt{\delta}(\delta-L)^{-\frac{1}{2}}X^\epsilon_\lambda(s)|_2^2ds
\nonumber\\
\!\!\!\!\!\!\!\!&&+\int_0^t\|\sqrt{\delta}(\delta-L)^{-\frac{1}{2}}B(s,X^\epsilon_\lambda(s),\mu^\epsilon_\lambda(s))\|_{L_2(U,V)}^2ds
\nonumber\\
\!\!\!\!\!\!\!\!&&+2\int_0^t\langle\sqrt{\delta}(\delta-L)^{-\frac{1}{2}}B(s,X^\epsilon_\lambda(s),\mu^\epsilon_\lambda(s))dW(s),\sqrt{\delta}(\delta-L)^{-\frac{1}{2}}X^\epsilon_\lambda(s)\rangle_2.
\end{eqnarray*}
By Burkholder-Davis-Gundy's inequality, it follows that
\begin{eqnarray}\label{32}
\!\!\!\!\!\!\!\!&&\mathbb{E}\left[\sup_{s\in[0,t]}|\sqrt{\delta}(\delta-L)^{-\frac{1}{2}}X^\epsilon_\lambda(s)|_2^2\right]+2\lambda\mathbb{E}\int_0^t \|\sqrt{\delta}(\delta-L)^{-\frac{1}{2}}X^\epsilon_\lambda(s)\|_{F_{1,2}}^2ds
\nonumber\\
\leq\!\!\!\!\!\!\!\!&&\mathbb{E}|\sqrt{\delta}(\delta-L)^{-\frac{1}{2}}X(0)|_2^2+2\mathbb{E}\int_0^t |X^\epsilon_\lambda(s)|_2^2ds+\mathbb{E}\int_0^t\|B(s,X^\epsilon_\lambda(s),\mu^\epsilon_\lambda(s))\|_{L_2(U,V)}^2ds
\nonumber\\
\!\!\!\!\!\!\!\!&&+8\mathbb{E}\Big(\int_0^t\|\sqrt{\delta}(\delta-L)^{-\frac{1}{2}}B(s,X^\epsilon_\lambda(s),\mu^\epsilon_\lambda(s))\|_{L_2(U,V)}^2|\sqrt{\delta}(\delta-L)^{-\frac{1}{2}}X^\epsilon_\lambda(s)|_2^2ds\Big)^{\frac{1}{2}}
\nonumber\\
\leq\!\!\!\!\!\!\!\!&&\mathbb{E}|\sqrt{\delta}(\delta-L)^{-\frac{1}{2}}X(0)|_2^2+\frac{1}{2}\mathbb{E}\left[\sup_{s\in[0,t]}|\sqrt{\delta}(\delta-L)^{-\frac{1}{2}}X^\epsilon_\lambda(s)|_2^2\right]
\nonumber\\
\!\!\!\!\!\!\!\!&&+C_1\mathbb{E}\int_0^t\Big[|X^\epsilon_\lambda(s)|_{2}^2+\mu^\epsilon_\lambda(s)(\|\cdot\|_{F_{1,2}^*}^2)\Big]ds+KT,
\end{eqnarray}
where we used the contraction of $\sqrt{\delta}(\delta-L)^{-\frac{1}{2}}$ on $L^2(\mu_M)$ and ({\textbf{A}}{\textbf{4}}), and the constants $C_1,K>0$ are independent of $\epsilon$ and $\lambda$. And we would like to remark that since $|\sqrt{\delta}(\delta-L)^{-\frac{1}{2}}\cdot|_2$ is equivalent to $\|\cdot\|_{F_{1,2}^*}$, the second term of the right hand side of (\ref{32}) is finite by $X^\epsilon_\lambda\in L^2(\Omega;C([0,T];H))$. Rearranging this inequality leads to
\begin{eqnarray}\label{14}
\!\!\!\!\!\!\!\!&&\mathbb{E}\left[\sup_{s\in[0,t]}|\sqrt{\delta}(\delta-L)^{-\frac{1}{2}}X^\epsilon_\lambda(s)|_2^2\right]+4\lambda\mathbb{E}\int_0^t |(1-L)^{\frac{1}{2}}\sqrt{\delta}(\delta-L)^{-\frac{1}{2}}X^\epsilon_\lambda(s)|_{2}^2ds
\nonumber\\
\leq\!\!\!\!\!\!\!\!&&2\mathbb{E}|X(0)|_2^2+2KT+C_2\int_0^t\mathbb{E}\left[\sup_{r\in[0,s]}|X^\epsilon_\lambda(r)|_2^2\right]dr,
\end{eqnarray}
where we used the embedding $V\subset H$ is continuous, and the constant $C_2$ is independent of $\epsilon$ and $\lambda$. Furthermore, the left hand side of (\ref{14}) is an increasing function under $\sup_{s\in[0,t]}$~w.r.t.~$\delta$ so that taking $\delta\to\infty$ yields that
\begin{eqnarray*}
\!\!\!\!\!\!\!\!&&\mathbb{E}\left[\sup_{s\in[0,t]}|X^\epsilon_\lambda(s)|_2^2\right]+4\lambda\mathbb{E}\int_0^t \|X^\epsilon_\lambda(s)\|_{F_{1,2}}^2ds
\nonumber\\
\leq\!\!\!\!\!\!\!\!&&2\mathbb{E}|X(0)|_2^2+2KT+C_2\int_0^t\mathbb{E}\left[\sup_{r\in[0,s]}|X^\epsilon_\lambda(r)|_2^2\right]dr.
\end{eqnarray*}
By Gronwall's lemma  we have
\begin{equation*}
\mathbb{E}\left[\sup_{s\in[0,t]}|X^\epsilon_\lambda(s)|_2^2\right]+4\lambda\mathbb{E}\int_0^t \|X^\epsilon_\lambda(s)\|_{F_{1,2}}^2ds\leq(2\mathbb{E}|X(0)|_2^2+2KT)e^{C_2T},~~t\in[0,T].
\end{equation*}
The continuity of $X^\epsilon_\lambda$ on $L^2(\mu_M)$ is a direct consequence of \cite[Theorem 2.1]{K1}.
\end{proof}

\textbf{Continuation of Proof of  Theorem \ref{t2}}

\textbf{(Existence)} Now let us complete the proof of Theorem \ref{t2} by verifying $\{X^\epsilon_\lambda\}_{\lambda\in(0,1)}$ weakly converges to $X^\epsilon$ in $L^2(\Omega\times[0,T];L^2(\mu_M))$ as $\lambda\to 0$.

Firstly, by making use of It\^{o}'s formula, for any $\lambda,\tilde{\lambda}\in(0,1)$ and $t\in[0,T]$,
\begin{eqnarray}\label{36}
\!\!\!\!\!\!\!\!&&\|X^\epsilon_\lambda(t)-X^\epsilon_{\tilde{\lambda}}(t)\|_{F_{1,2}^*}^2
\nonumber\\
\!\!\!\!\!\!\!\!&&+2\int_0^t\langle\Psi(s,X^\epsilon_\lambda(s),\mu^\epsilon_\lambda(s))-\Psi(s,X^\epsilon_{\tilde{\lambda}}(s),\mu^\epsilon_{\tilde{\lambda}}(s))+\lambda X^\epsilon_\lambda(s)-\tilde{\lambda}X^\epsilon_{\tilde{\lambda}}(s), X^\epsilon_\lambda(s)-X^\epsilon_{\tilde{\lambda}}(s)\rangle_2ds
\nonumber\\
=\!\!\!\!\!\!\!\!&&2\int_0^t(1-\epsilon)\langle\Psi(s,X^\epsilon_\lambda(s),\mu^\epsilon_\lambda(s))-\Psi(s,X^\epsilon_{\tilde{\lambda}}(s),\mu^\epsilon_{\tilde{\lambda}}(s))+\lambda X^\epsilon_\lambda(s)-\tilde{\lambda}X^\epsilon_{\tilde{\lambda}}(s), X^\epsilon_\lambda(s)-X^\epsilon_{\tilde{\lambda}}(s)\rangle_{F_{1,2}^*}ds
\nonumber\\
\!\!\!\!\!\!\!\!&&+\int_0^t\|B(s,X^\epsilon_\lambda(s),\mu^\epsilon_\lambda(s))-B(s,X^\epsilon_{\tilde{\lambda}}(s),\mu^\epsilon_{\tilde{\lambda}}(s))\|_{L_2(U,H)}^2ds
\nonumber\\
\!\!\!\!\!\!\!\!&&+2\int_0^t\langle \big(B (s,X^\epsilon_\lambda(s),\mu^\epsilon_\lambda(s))-B(s,X^\epsilon_{\tilde{\lambda}}(s),\mu^\epsilon_{\tilde{\lambda}}(s))\big)dW(s),X^\epsilon_\lambda(s)-X^\epsilon_{\tilde{\lambda}}(s)\rangle_{F_{1,2}^*},
\end{eqnarray}
here we denote $\mu^\epsilon_{\tilde{\lambda}}(s):=\mathcal{L}_{X^\epsilon_{\tilde{\lambda}}(s)}$.

From ({\textbf{A}}{\textbf{3}}) we obtain
\begin{eqnarray*}
\!\!\!\!\!\!\!\!&&2\int_0^t\langle\Psi(s,X^\epsilon_\lambda(s),\mu^\epsilon_\lambda(s))-\Psi(s,X^\epsilon_{\tilde{\lambda}}(s),\mu^\epsilon_{\tilde{\lambda}}(s))+\lambda X^\epsilon_\lambda(s)-\tilde{\lambda}X^\epsilon_{\tilde{\lambda}}(s), X^\epsilon_\lambda(s)-X^\epsilon_{\tilde{\lambda}}(s)\rangle_2ds
\nonumber\\
\geq\!\!\!\!\!\!\!\!&&\alpha_1\int_0^t|\Psi(s,X^\epsilon_\lambda(s),\mu^\epsilon_\lambda(s))-\Psi(s,X^\epsilon_{\tilde{\lambda}}(s),\mu^\epsilon_{\tilde{\lambda}}(s))|_2^2ds-\alpha_2\int_0^t\mathbb{W}_{2,H}
(\mu^\epsilon_\lambda(s),\mu^\epsilon_{\tilde{\lambda}}(s))^2ds
\nonumber\\
\!\!\!\!\!\!\!\!&&-\alpha_3\int_0^t\|X^\epsilon_\lambda(s)-X^\epsilon_{\tilde{\lambda}}(s)\|_{F_{1,2}^*}^2ds+2\int_0^t\langle\lambda X^\epsilon_\lambda(s)-\tilde{\lambda}X^\epsilon_{\tilde{\lambda}}(s), X^\epsilon_\lambda(s)-X^\epsilon_{\tilde{\lambda}}(s)\rangle_2ds.
\end{eqnarray*}
The first integral of the right hand side of (\ref{36}) is controlled by
\begin{eqnarray*}
\!\!\!\!\!\!\!\!&&2\int_0^t(1-\epsilon)\langle\Psi(s,X^\epsilon_\lambda(s),\mu^\epsilon_\lambda(s))-\Psi(s,X^\epsilon_{\tilde{\lambda}}(s),\mu^\epsilon_{\tilde{\lambda}}(s))+\lambda X^\epsilon_\lambda(s)-\tilde{\lambda}X^\epsilon_{\tilde{\lambda}}(s), X^\epsilon_\lambda(s)-X^\epsilon_{\tilde{\lambda}}(s)\rangle_{F_{1,2}^*}ds
\nonumber\\
\leq\!\!\!\!\!\!\!\!&&\varepsilon_0\int_0^t|\Psi(s,X^\epsilon_\lambda(s),\mu^\epsilon_\lambda(s))-\Psi(s,X^\epsilon_{\tilde{\lambda}}(s),\mu^\epsilon_{\tilde{\lambda}}(s))|_2^2ds
+C_{\varepsilon_0}\int_0^t\|X^\epsilon_\lambda(s)-X^\epsilon_{\tilde{\lambda}}(s)\|_{F_{1,2}^*}^2ds
\nonumber\\
\!\!\!\!\!\!\!\!&&+C_1(\lambda+\tilde{\lambda})\int_0^t(|X^\epsilon_\lambda(s)|_2^2+|X^\epsilon_{\tilde{\lambda}}(s)|_2^2)ds,
\end{eqnarray*}
where we used Young's inequality and take $\varepsilon_0<\alpha_1$, the constants $C_{\varepsilon_0},C_1>0$ are independent of $\epsilon$ and $\lambda$.

The Burkholder-Davis-Gundy's inequality gives that
\begin{eqnarray*}
\!\!\!\!\!\!\!\!&&\mathbb{E}\left[\sup_{s\in[0,t]}\|X^\epsilon_\lambda(s)-X^\epsilon_{\tilde{\lambda}}(s)\|_{F_{1,2}^*}^2\right]
+(\alpha_1-\varepsilon_0)\mathbb{E}\int_0^t|\Psi(s,X^\epsilon_\lambda(s),\mu^\epsilon_\lambda(s))-\Psi(s,X^\epsilon_{\tilde{\lambda}}(s),\mu^\epsilon_{\tilde{\lambda}}(s))|_2^2ds
\nonumber\\
\leq\!\!\!\!\!\!\!\!&&C_2\int_0^t\mathbb{E}\|X^\epsilon_\lambda(s)-X^\epsilon_{\tilde{\lambda}}(s)\|_{F_{1,2}^*}^2ds+C_3(\lambda+\tilde{\lambda})\int_0^t(|X^\epsilon_\lambda(s)|_2^2+|X^\epsilon_{\tilde{\lambda}}(s)|_2^2)ds
\nonumber\\
\!\!\!\!\!\!\!\!&&+8\mathbb{E}\Big(\int_0^t\|B (s,X^\epsilon_\lambda(s),\mu^\epsilon_\lambda(s))-B(s,X^\epsilon_{\tilde{\lambda}}(s),\mu^\epsilon_{\tilde{\lambda}}(s))\|_{L_2(U,H)}^2\|X^\epsilon_\lambda(s)-X^\epsilon_{\tilde{\lambda}}(s)\|_{F_{1,2}^*}^2ds\Big)^{\frac{1}{2}}
\nonumber\\
\leq\!\!\!\!\!\!\!\!&&\frac{1}{2}\mathbb{E}\left[\sup_{s\in[0,t]}\|X^\epsilon_\lambda(s)-X^\epsilon_{\tilde{\lambda}}(s)\|_{F_{1,2}^*}^2\right]+C_4\int_0^t\mathbb{E}\left[\sup_{r\in[0,s]}\|X^\epsilon_\lambda(r)-X^\epsilon_{\tilde{\lambda}}(r)\|_{F_{1,2}^*}^2\right]ds
\nonumber\\
\!\!\!\!\!\!\!\!&&+C_3(\lambda+\tilde{\lambda})\mathbb{E}\int_0^t(|X^\epsilon_\lambda(s)|_2^2+|X^\epsilon_{\tilde{\lambda}}(s)|_2^2)ds.
\end{eqnarray*}
where the constants $C_2,C_3,C_4>0$ are independent of $\epsilon$ and $\lambda$. Then Lemma \ref{l3} and Gronwall's lemma imply that there is a constant $C$ independent of $\epsilon$ and $\lambda$,
\begin{eqnarray}\label{15}
\!\!\!\!\!\!\!\!&&\mathbb{E}\left[\sup_{s\in[0,t]}\|X^\epsilon_\lambda(s)-X^\epsilon_{\tilde{\lambda}}(s)\|_{F_{1,2}^*}^2\right]
+2(\alpha_1-\varepsilon_0)\mathbb{E}\int_0^t|\Psi(s,X^\epsilon_\lambda(s),\mu^\epsilon_\lambda(s))-\Psi(s,X^\epsilon_{\tilde{\lambda}}(s),\mu^\epsilon_{\tilde{\lambda}}(s))|_2^2ds
\nonumber\\
\leq\!\!\!\!\!\!\!\!&&C(\lambda+\tilde{\lambda}).
\end{eqnarray}

Hence $\{X^\epsilon_\lambda\}$ is a Cauchy net in $L^2(\Omega;C([0,T];F_{1,2}^*))$ with respect to $\lambda$. Taking $\lambda\to 0$, (\ref{15}) gives that there is a continuous $(\mathcal{F}_t)_{t\geq 0}$-adapted $F_{1,2}^*$-valued process $\{X^\epsilon(t)\}_{t\in[0,T]}$ such that $X^\epsilon_\lambda\to X^\epsilon$ strongly in $L^2(\Omega;C([0,T];F_{1,2}^*))$. Furthermore, thanks to the Banach-Steinhaus theorem, Lemma \ref{l3} and $F_{1,2}^*\subset(L^2(\mu_M))^*$ densely imply that $X^\epsilon_\lambda\to X^\epsilon$ weakly in $L^2(\Omega\times[0,T];L^2(\mu_M))$ as $\lambda\to 0$. And for the distribution we infer that
$$\mathbb{W}_{2,H}(\mathcal{L}_{X^\epsilon_{\lambda}(t)},\mathcal{L}_{X^\epsilon(t)})^2\leq\mathbb{E}\|X^\epsilon_\lambda(t)-X^\epsilon(t)\|_{F_{1,2}^*}^2\downarrow 0~\text{as}~\lambda\downarrow 0.$$
This together with Burkholder-Davis-Gundy's inequality yields that
$\int_0^{\cdot}B(s,X^\epsilon_\lambda(s),\mathcal{L}_{X^\epsilon_{\lambda}(s)})dW(s)$ also converges strongly to $\int_0^{\cdot}B(s,X^\epsilon(s),\mathcal{L}_{X^\epsilon(s)})dW(s)$ in $L^2(\Omega;C([0,T];F_{1,2}^*))$ as $\lambda\to 0$. Moreover, Eq.~(\ref{11}) and Lemma \ref{l1} imply that
$$\int_0^\cdot\Psi(s,X^\epsilon_\lambda(s),\mathcal{L}_{X^\epsilon_{\lambda}(s)})+\lambda X^\epsilon_\lambda(s)ds$$
converges strongly to an element in $L^2(\Omega;C([0,T];F_{1,2}))$. Due to Lemma \ref{l3} and (\ref{15}), it is obvious that $\{\Psi(\cdot,X^\epsilon_\lambda(\cdot),\mathcal{L}_{X^\epsilon_{\lambda}(\cdot)})+\lambda X^\epsilon_\lambda(\cdot)\}_{\lambda\in(0,1)}$  strongly converges to an element denoted by $Y(\cdot)$ in $L^2(\Omega\times[0,T];L^2(\mu_M))$. However, it is not easy to check  $\Psi(\cdot,X^\epsilon_\lambda(\cdot),\mathcal{L}_{X^\epsilon_{\lambda}(\cdot)})+\lambda X^\epsilon_\lambda(\cdot)$ strongly converges to $\Psi(\cdot,X^\epsilon(\cdot),\mathcal{L}_{X^\epsilon(\cdot)})$ in $L^2(\Omega\times[0,T];L^2(\mu_M))$ directly, therefore, we consider the weak limit instead.

Recalling for any $t\in[0,T]$, the equation
\begin{equation*}
X^\epsilon_\lambda(t)+(\epsilon-L)\int_0^t\big(\Psi(s,X^\epsilon_\lambda(s),\mathcal{L}_{X^\epsilon_\lambda(s)})+\lambda X^\epsilon_\lambda(s)\big)ds=X(0)+\int_0^tB(s,X^\epsilon_\lambda(s),\mathcal{L}_{X^\epsilon_\lambda(s)})dW(s).
\end{equation*}
Following the convergence arguments above, taking $\lambda\to 0$ and by Lemma \ref{l2}, it is obvious that for any $t\in[0,T]$,
\begin{equation*}
X^\epsilon(t)+\int_0^t (\epsilon-L) Y(s)ds=X(0)+\int_0^tB(s,X^\epsilon(s),\mathcal{L}_{X^\epsilon(s)})dW(s)~\text{holds in}~ (L^2(\mu_M))^*.
\end{equation*}

We now aim to prove $Y(\cdot)=\Psi(X^\epsilon(\cdot),\mathcal{L}_{X^\epsilon(\cdot)})$, $dt\times\mathbb{P}$-a.s.
Using It\^{o}'s formula and the product rule, we have for any $c\geq 0$ that
\begin{eqnarray}\label{16}
\!\!\!\!\!\!\!\!&&\mathbb{E}\big[e^{-ct}\|X^\epsilon(t)\|_{F_{1,2}^*}^2\big]-\mathbb{E}\|X(0)\|_{F_{1,2}^*}^2
\nonumber\\
=\!\!\!\!\!\!\!\!&&\mathbb{E}\int_0^te^{-cs}\big(2~_{V^*}\langle(L-\epsilon)Y(s), X^\epsilon(s)\rangle_V+\|B(s,X^\epsilon(s),\mu^\epsilon(s))\|_{L_2(U,H)}^2-c\|X^\epsilon(s)\|_{F_{1,2}^*}^2\big)ds,
\nonumber\\
\end{eqnarray}
where we remain use $\mu^\epsilon(s)=\mathcal{L}_{X^\epsilon(s)}$ for simplicity.

For any $\phi\in L^2([0,T]\times\Omega;L^2(\mu_M))\cap L^2(\Omega;C([0,T];F_{1,2}^*))$, using It\^{o}'s formula yields that
\begin{eqnarray}\label{17}
\!\!\!\!\!\!\!\!&&\mathbb{E}\big[e^{-ct}\|X^\epsilon_\lambda(t)\|_{F_{1,2}^*}^2\big]-\mathbb{E}\|X(0)\|_{F_{1,2}^*}^2
\nonumber\\
\leq\!\!\!\!\!\!\!\!&&\mathbb{E}\Big\{\int_0^te^{-cs}\Big[\|B(s,X^\epsilon_\lambda(s),\mu^\epsilon_\lambda(s))-B(s,\phi(s),\mu_\phi(s))\|_{L_2(U,H)}^2-c\|X^\epsilon_\lambda(s)-\phi(s)\|_{F_{1,2}^*}^2
\nonumber\\
\!\!\!\!\!\!\!\!&&+2~_{V^*}\langle(L-\epsilon)\big[\big(\Psi(s,X^\epsilon_\lambda(s),\mu^\epsilon_\lambda(s))+\lambda X^\epsilon_\lambda(s)\big)
\nonumber\\
\!\!\!\!\!\!\!\!&&
-\big(\Psi(s,\phi(s),\mu_\phi(s))+\lambda \phi(s)\big)\big],X^\epsilon_\lambda(s)- \phi(s)\rangle_V\Big]ds\Big\}
\nonumber\\
\!\!\!\!\!\!\!\!&&+\mathbb{E}\Big\{\int_0^te^{-cs}\Big[2~_{V^*}\langle(L-\epsilon)\big[\Psi(s,\phi(s),\mu_\phi(s))+\lambda\phi(s)\big],X^\epsilon_\lambda(s)\rangle_V-2c\langle X^\epsilon_\lambda(s),\phi(s)\rangle_{F_{1,2}^*}
\nonumber\\
\!\!\!\!\!\!\!\!&&+c\|\phi(s)\|_{F_{1,2,\epsilon}^*}^2+2~_{V^*}\langle(L-\epsilon)\big[\big(\Psi(s,X^\epsilon_\lambda(s),\mu^\epsilon_\lambda(s))+\lambda X^\epsilon_\lambda(s)\big)
\nonumber\\
\!\!\!\!\!\!\!\!&&
-\big(\Psi(s,\phi(s),\mu_\phi(s))+\lambda \phi(s)\big)\big],\phi(s)\rangle_{V}
\nonumber\\
\!\!\!\!\!\!\!\!&&+2\langle B (s,X^\epsilon_\lambda(s),\mu^\epsilon_\lambda(s)),B(s,\phi(s),\mu_\phi(s))\rangle_{L_2(U,H)}-\|B(s,\phi(s),\mu_\phi(s))\|_{L_2(U,H)}^2\Big]ds\Big\},~~
\end{eqnarray}
where we denote $\mu_{\phi}(s):=\mathcal{L}_{\phi(s)}$. The first integral of the right hand side of (\ref{17}) can be controlled by using (\ref{es1}) and taking $c=\alpha_2+\alpha_3+2K_1+C_{\epsilon,\lambda}$ that
\begin{eqnarray}\label{18}
\!\!\!\!\!\!\!\!&&\mathbb{E}\Big\{\int_0^te^{-cs}\Big[\|B(s,X^\epsilon_\lambda(s),\mu^\epsilon_\lambda(s))-B(s,\phi(s),\mu_\phi(s))\|_{L_2(U,H)}^2-c\|X^\epsilon_\lambda(s)-\phi(s)\|_{F_{1,2}^*}^2
\nonumber\\
\!\!\!\!\!\!\!\!&&+2~_{V^*}\langle(L-\epsilon)\big[\big(\Psi(s,X^\epsilon_\lambda(s),\mu^\epsilon_\lambda(s))+\lambda X^\epsilon_\lambda(s)\big)
\nonumber\\
\!\!\!\!\!\!\!\!&&
-\big(\Psi(s,\phi(s),\mu_\phi(s))+\lambda \phi(s)\big)\big],X^\epsilon_\lambda(s)- \phi(s)\rangle_V\Big]ds\Big\}
\nonumber\\
\leq\!\!\!\!\!\!\!\!&&\mathbb{E}\Big\{\int_0^te^{-cs}\Big[
(\alpha_2+K_1)\mathbb{W}_{2,H}(\mu^\epsilon_\lambda(s),\mu_\phi(s))^2+(C_{\epsilon,\lambda}+\alpha_3+K_1)\|X^\epsilon_\lambda(s)-\phi(s)\|_{F_{1,2}^*}^2
\nonumber\\
\!\!\!\!\!\!\!\!&&-c\|X^\epsilon_\lambda(s)-\phi(s)\|_{F_{1,2}^*}^2\Big]ds\Big\}
\nonumber\\
\leq\!\!\!\!\!\!\!\!&&\mathbb{E}\Big\{\int_0^te^{-cs}\Big[
(\alpha_2+\alpha_3+2K_1+C_{\epsilon,\lambda})\|X^\epsilon_\lambda(s)-\phi(s)\|_{F_{1,2}^*}^2-c\|X^\epsilon_\lambda(s)-\phi(s)\|_{F_{1,2}^*}^2\Big]ds\Big\}
\nonumber\\
=\!\!\!\!\!\!\!\!&&0.
\end{eqnarray}
Combining (\ref{17}) with (\ref{18}), for any non-negative $\varphi\in L^\infty([0,T],dt;\mathbb{R})$,
\begin{eqnarray}\label{19}
\!\!\!\!\!\!\!\!&&\mathbb{E}\Big[\int_0^T\varphi(t)\Big(e^{-ct}\|X^\epsilon(t)\|_{F_{1,2}^*}^2-\|X(0)\|_{F_{1,2}^*}^2\Big)dt\Big]
\nonumber\\
\leq\!\!\!\!\!\!\!\!&&\liminf_{\lambda\to 0}\mathbb{E}\Big[\int_0^T\varphi(t)\Big(e^{-ct}\|X^\epsilon_\lambda(t)\|_{F_{1,2}^*}^2-\|X(0)\|_{F_{1,2}^*}^2\Big)dt\Big]
\nonumber\\
\leq\!\!\!\!\!\!\!\!&&\mathbb{E}\Big\{\int_0^T\varphi(t)\Big[\int_0^te^{-cs}\Big(2~_{V^*}\langle(L-\epsilon)\Psi(s,\phi(s),\mu_\phi(s)),X^\epsilon(s)\rangle_V-2c\langle X^\epsilon(s),\phi(s)\rangle_{F_{1,2}^*}
\nonumber\\
\!\!\!\!\!\!\!\!&&
+c\|\phi(s)\|_{F_{1,2}^*}^2+2~_{V^*}\langle(L-\epsilon)\big[Y(s) -\Psi(s,\phi(s),\mu_\phi(s))\big],\phi(s)\rangle_{V}
\nonumber\\
\!\!\!\!\!\!\!\!&&+2\langle B (s,X^\epsilon(s),\mu^\epsilon(s)),B(s,\phi(s),\mu_\phi(s))\rangle_{L_2(U,H)}-\|B(s,\phi(s),\mu_\phi(s))\|_{L_2(U,H)}^2\Big)ds\Big]dt\Big\}.
\nonumber\\
\end{eqnarray}
Taking (\ref{16}) into the left hand side of (\ref{19}) and then rearranging (\ref{19}), it leads to
\begin{eqnarray}\label{20}
\!\!\!\!\!\!\!\!&&\mathbb{E}\Big\{\int_0^T\varphi(t)\Big[\int_0^te^{-cs}\Big(2~_{V^*}\langle(L-\epsilon)Y(s)-(L-\epsilon)\Psi(s,\phi(s),\mu_\phi(s)),X^\epsilon(s)-\phi(s)\rangle_V
\nonumber\\
\!\!\!\!\!\!\!\!&&-c\|X^\epsilon(s)-\phi(s)\|_{F_{1,2}^*}^2\Big)ds\Big]dt\Big\}\leq 0.
\end{eqnarray}
Letting $\phi=X^\epsilon-\eta\tilde{\phi}v$ for any $\eta>0$, $v\in L^2(\mu_M)$ and $\tilde{\phi}\in L^\infty([0,T]\times\Omega,dt\times\mathbb{P};\mathbb{R})$. Splitting both sides of (\ref{20}) by $\eta$, and it follows that
$$\mathbb{W}_{2,H}(\mu^\epsilon(s),\mu_\phi(s))^2\leq\mathbb{E}\|\eta\tilde{\phi}(s)v\|_{F_{1,2}^*}^2\leq\eta\|\tilde{\phi}\|_{\infty}^2\|v\|_{F_{1,2}^*}^2\downarrow0,~\text{as}~\eta\downarrow0.$$
Then taking $\eta\to 0$, by Lebesgue's dominated convergence theorem, we get
\begin{eqnarray*}
\mathbb{E}\Big\{\int_0^T\varphi(t)\Big[\int_0^te^{-cs}\Big(2~_{V^*}\langle(L-\epsilon)Y(s)-(L-\epsilon)\Psi(s,X^\epsilon(s),\mu^\epsilon(s)),\tilde{\phi}(s)v\rangle_V
ds\Big]dt\Big\}\leq 0.
\end{eqnarray*}
Replacing $\tilde{\phi}$ with $-\tilde{\phi}$, it follows that
\begin{eqnarray*}
\mathbb{E}\Big\{\int_0^T\varphi(t)\Big[\int_0^te^{-cs}\Big(2~_{V^*}\langle(L-\epsilon)Y(s)-(L-\epsilon)\Psi(s,X^\epsilon(s),\mu^\epsilon(s)),\tilde{\phi}(s)v\rangle_V
ds\Big]dt\Big\}=0.
\end{eqnarray*}
Since $\varphi,\tilde{\phi},v$ are arbitrary, we finally conclude that $Y(\cdot)=\Psi(\cdot,X^\epsilon(\cdot),\mu^\epsilon(\cdot))$, $dt\times\mathbb{P}$-a.s., which combines with (\ref{15}) that
$$\int_0^\cdot\Psi(s,X^\epsilon_\lambda(s),\mu^\epsilon_{\lambda}(s))+\lambda X^\epsilon_\lambda(s)ds\to\int_0^\cdot\Psi(s,X^\epsilon(s),\mu^\epsilon(s))ds~\text{strongly in}~L^2(\Omega\times[0,T];L^2(\mu_M)).$$
Consequently,
$$\int_0^\cdot\Psi(s,X^\epsilon_\lambda(s),\mu^\epsilon_{\lambda}(s))+\lambda X^\epsilon_\lambda(s)ds\to\int_0^\cdot\Psi(s,X^\epsilon(s),\mu^\epsilon(s))ds~\text{strongly in}~L^2(\Omega;C([0,T];F_{1,2})),$$
which yields that $\int_0^\cdot\Psi(s,X^\epsilon(s),\mu^\epsilon(s))ds\in C([0,T];F_{1,2})$, $dt\times\mathbb{P}$-a.s.

Furthermore, (\ref{6}) is a consequent result following from Lemma \ref{l3} and lower semi-continuity.

\textbf{(Uniqueness)} Let $X^\epsilon,Y^\epsilon$ be two solutions to Eq.~(\ref{e3}) with $X^\epsilon(0)=X(0)\in L^2(\Omega,\mathbb{P};H)$ and $Y^\epsilon(0)=Y(0)\in L^2(\Omega,\mathbb{P};H)$, then $\mathbb{P}$-a.s.
\begin{eqnarray*}
\!\!\!\!\!\!\!\!&&X^\epsilon(t)-Y^\epsilon(t)+(\epsilon-L)\int_0^t\big(\Psi(s,X^\epsilon(s),\mathcal{L}_{X^\epsilon(s)})-\Psi(s,Y^\epsilon(s),\mathcal{L}_{Y^\epsilon(s)})\big)ds
\nonumber\\
=\!\!\!\!\!\!\!\!&&
(X(0)-Y(0))+\int_0^t\big(B(s,X^\epsilon(s),\mathcal{L}_{X^\epsilon(s)})-B(s,Y^\epsilon(s),\mathcal{L}_{Y^\epsilon(s)})\big)dW(s),~t\in[0,T].
\end{eqnarray*}
Using It\^{o}'s formula to $\|X^\epsilon(t)-Y^\epsilon(t)\|_{F_{1,2}^*}^2$ and by Lemma \ref{l2} that
\begin{eqnarray}\label{21}
\!\!\!\!\!\!\!\!&&\|X^\epsilon(t)-Y^\epsilon(t)\|_{F_{1,2}^*}^2+2\int_0^t\langle\Psi(s,X^\epsilon(s),\mathcal{L}_{X^\epsilon(s)})-\Psi(s,Y^\epsilon(s),\mathcal{L}_{Y^\epsilon(s)}),X^\epsilon(s)-Y^\epsilon(s)\rangle_2ds
\nonumber\\
=\!\!\!\!\!\!\!\!&&\|X(0)-Y(0)\|_{F_{1,2}^*}^2+\int_0^t\|B(s,X^\epsilon(s),\mathcal{L}_{X^\epsilon(s)})-B(s,Y^\epsilon(s),\mathcal{L}_{Y^\epsilon(s)})\|_{L_2(U,H)}^2ds
\nonumber\\
\!\!\!\!\!\!\!\!&&+2\int_0^t\big\langle\big( B(s,X^\epsilon(s),\mathcal{L}_{X^\epsilon(s)})-B(s,Y^\epsilon(s),\mathcal{L}_{Y^\epsilon(s)})\big)dW(s),X^\epsilon(s)-Y^\epsilon(s)\big\rangle_{F_{1,2}^*}
\nonumber\\
\!\!\!\!\!\!\!\!&&
+2(1-\epsilon)\int_0^t\langle\Psi(s,X^\epsilon(s),\mathcal{L}_{X^\epsilon(s)})-\Psi(s,Y^\epsilon(s),\mathcal{L}_{Y^\epsilon(s)}),X^\epsilon(s)-Y^\epsilon(s)\rangle_{F_{1,2}^*}ds.
\end{eqnarray}
Following from ({\textbf{A}}{\textbf{2}}), ({\textbf{A}}{\textbf{3}}) and (\ref{es1}), and taking expectation on both sides of (\ref{21}) that
\begin{eqnarray*}
\!\!\!\!\!\!\!\!&&\mathbb{E}\|X^\epsilon(t)-Y^\epsilon(t)\|_{F_{1,2}^*}^2
\leq \mathbb{E}\|X(0)-Y(0)\|_{F_{1,2}^*}^2+C\int_0^t\mathbb{E}\|X^\epsilon(s)-Y^\epsilon(s)\|_{F_{1,2}^*}^2ds.
\end{eqnarray*}
Consequently, by Gronwall's lemma, if $X(0)=Y(0)$, then we have $X^\epsilon=Y^\epsilon$, $\mathbb{P}$-a.s., which implies the uniqueness. The proof of Theorem \ref{t2} is complete.\hspace{\fill}$\Box$

\subsection{Proof of Theorem \ref{t1}}
The main idea is to verify the sequence $\{X^\epsilon\}_{\epsilon\in(0,1)}$ defined in Eq.~(\ref{e2}) converges to the solution of (\ref{e1}) when $\epsilon\to 0$.

\textbf{(Existence)} Firstly, applying It\^{o}'s formula and by Lemma \ref{l2} we have
\begin{eqnarray}\label{22}
\!\!\!\!\!\!\!\!&&\|X^\epsilon(t)\|_{F_{1,2}^*}^2+2\int_0^t\langle\Psi(s,X^\epsilon(s),\mu^\epsilon(s)),X^\epsilon(s)\rangle_2ds
\nonumber\\
=\!\!\!\!\!\!\!\!&&\|X(0)\|_{F_{1,2}^*}^2+2(1-\epsilon)\int_0^t\langle\Psi(s,X^\epsilon(s),\mu^\epsilon(s)),X^\epsilon(s)\rangle_{F_{1,2}^*}ds+\int_0^t\|B(s,X^\epsilon(s),\mu^\epsilon(s))\|_{L_2(U,H)}^2ds
\nonumber\\
\!\!\!\!\!\!\!\!&&+2\int_0^t\big\langle B(s,X^\epsilon(s),\mu^\epsilon(s))dW(s),X^\epsilon(s)\big\rangle_{F_{1,2}^*}.
\end{eqnarray}
Taking expectation to both sides of (\ref{22}), then it follows from ({\textbf{A}}{\textbf{3}}) that
\begin{eqnarray*}
\!\!\!\!\!\!\!\!&&\mathbb{E}\|X^\epsilon(t)\|_{F_{1,2}^*}^2+\alpha_1\mathbb{E}\int_0^t|\Psi(s,X^\epsilon(s),\mu^\epsilon(s))|_2^2ds
\nonumber\\
\leq\!\!\!\!\!\!\!\!&&\mathbb{E}\|X(0)\|_{F_{1,2}^*}^2+2(1-\epsilon)\mathbb{E}\int_0^t\|\Psi(s,X^\epsilon(s),\mu^\epsilon(s))\|_{F_{1,2}^*}\|X^\epsilon(s)\|_{F_{1,2}^*}ds
\nonumber\\
\!\!\!\!\!\!\!\!&&+\mathbb{E}\int_0^t\|B(s,X^\epsilon(s),\mu^\epsilon(s))\|_{L_2(U,H)}^2ds+(\alpha_2+\alpha_3)\mathbb{E}\int_0^t\|X^\epsilon(s)\|_{F_{1,2}^*}^2ds
\nonumber\\
\leq\!\!\!\!\!\!\!\!&&\mathbb{E}\|X(0)\|_{F_{1,2}^*}^2+C_0(1-\epsilon)\mathbb{E}\int_0^t|\Psi(s,X^\epsilon(s),\mu^\epsilon(s))|_{2}\|X^\epsilon(s)\|_{F_{1,2}^*}ds
\nonumber\\
\!\!\!\!\!\!\!\!&&+(2K_1+\alpha_2+\alpha_3)\mathbb{E}\int_0^t\|X^\epsilon(s)\|_{F_{1,2}^*}^2ds
\nonumber\\
\leq\!\!\!\!\!\!\!\!&&\mathbb{E}\|X(0)\|_{F_{1,2}^*}^2+\varepsilon_0\mathbb{E}\int_0^t|\Psi(s,X^\epsilon(s),\mu^\epsilon(s))|_{2}^2ds+C\mathbb{E}\int_0^t\|X^\epsilon(s)\|_{F_{1,2}^*}^2ds,
\end{eqnarray*}
where the constants $C_0,C>0$ are independent of $\epsilon$ and $\varepsilon_0<\alpha_1$.

Then Gronwall's lemma implies that
\begin{eqnarray}\label{24}
\!\!\!\!\!\!\!\!&&\mathbb{E}\|X^\epsilon(t)\|_{F_{1,2}^*}^2+(\alpha_1-\varepsilon_0)\mathbb{E}\int_0^t|\Psi(s,X^\epsilon(s),\mu^\epsilon(s))|_2^2ds\leq e^{CT}\mathbb{E}\|X(0)\|_{F_{1,2}^*}^2,~t\in[0,T].~~~
\end{eqnarray}

Now we are in the position to show the convergence of solution. Applying It\^{o}'s formula for any $\epsilon,\tilde{\epsilon}\in(0,1)$ and $t\in[0,T]$,
\begin{eqnarray*}
\!\!\!\!\!\!\!\!&&\|X^\epsilon(t)-X^{\tilde{\epsilon}}(t)\|_{F_{1,2}^*}^2
+2\int_0^t\langle\Psi(s,X^\epsilon(s),\mu^\epsilon(s))-\Psi(s,X^{\tilde{\epsilon}}(s),\mu^{\tilde{\epsilon}}(s)), X^\epsilon(s)-X^{\tilde{\epsilon}}(s)\rangle_2ds
\nonumber\\
\leq\!\!\!\!\!\!\!\!&&2\int_0^t\langle\Psi(s,X^\epsilon(s),\mu^\epsilon(s))-\Psi(s,X^{\tilde{\epsilon}}(s),\mu^{\tilde{\epsilon}}(s)), X^\epsilon(s)-X^{\tilde{\epsilon}}(s)\rangle_{F_{1,2}^*}ds
\nonumber\\
\!\!\!\!\!\!\!\!&&+C_1\int_0^t\Big(\epsilon|\Psi(s,X^\epsilon(s),\mu^\epsilon(s))|_2+\tilde{\epsilon}|\Psi(s,X^{\tilde{\epsilon}}(s),\mu^{\tilde{\epsilon}}(s))|_2\Big) \|X^\epsilon(s)-X^{\tilde{\epsilon}}(s)\|_{F_{1,2}^*}ds
\nonumber\\
\!\!\!\!\!\!\!\!&&+K_1\int_0^t\|X^\epsilon(s)-X^{\tilde{\epsilon}}(s)\|_{F_{1,2}^*}^2+\mathbb{W}_{2,H}(\mu^\epsilon(s),\mu^{\tilde{\epsilon}}(s))^2ds
\nonumber\\
\!\!\!\!\!\!\!\!&&+2\int_0^t\langle \big(B (s,X^\epsilon(s),\mu^\epsilon(s))-B(s,X^{\tilde{\epsilon}}(s),\mu^{\tilde{\epsilon}}(s))\big)dW(s),X^\epsilon(s)-X^{\tilde{\epsilon}}(s)\rangle_{F_{1,2}^*},
\end{eqnarray*}
where we denote $\mu^{\tilde{\epsilon}}(s):=\mathcal{L}_{X^{\tilde{\epsilon}}(s)}$ and the constant $C_1>0$ independent of $\epsilon$.

Following from ({\textbf{A}}{\textbf{3}}), by Young's inequality, it leads to
\begin{eqnarray}\label{23}
\!\!\!\!\!\!\!\!&&\|X^\epsilon(t)-X^{\tilde{\epsilon}}(t)\|_{F_{1,2}^*}^2
+\alpha_1\int_0^t|\Psi(s,X^\epsilon(s),\mu^\epsilon(s))-\Psi(s,X^{\tilde{\epsilon}}(s),\mu^{\tilde{\epsilon}}(s))|_2^2ds
\nonumber\\
\leq\!\!\!\!\!\!\!\!&&\frac{\alpha_1}{2}\int_0^t|\Psi(s,X^\epsilon(s),\mu^\epsilon(s))-\Psi(s,X^{\tilde{\epsilon}}(s),\mu^{\tilde{\epsilon}}(s))|_2^2ds
\nonumber\\
\!\!\!\!\!\!\!\!&&+C_2(\epsilon+\tilde{\epsilon})\int_0^t|\Psi(s,X^\epsilon(s),\mu^\epsilon(s))|_2^2+|\Psi(s,X^{\tilde{\epsilon}}(s),\mu^{\tilde{\epsilon}}(s))|_2^2ds
\nonumber\\
\!\!\!\!\!\!\!\!&&+K_1\int_0^t\mathbb{W}_{2,H}(\mu^\epsilon(s),\mu^{\tilde{\epsilon}}(s))^2ds+C_3\int_0^t\|X^\epsilon(s)-X^{\tilde{\epsilon}}(s)\|_{F_{1,2}^*}^2ds
\nonumber\\
\!\!\!\!\!\!\!\!&&+2\int_0^t\langle \big(B (s,X^\epsilon(s),\mu^\epsilon(s))-B(s,X^{\tilde{\epsilon}}(s),\mu^{\tilde{\epsilon}}(s))\big)dW(s),X^\epsilon(s)-X^{\tilde{\epsilon}}(s)\rangle_{F_{1,2}^*},
\end{eqnarray}
where the constants $C_2,C_3>0$ are independent of $\epsilon,\tilde{\epsilon}$.

Taking expectation and rearranging (\ref{23}), the Burkholder-Davis-Gundy's inequality gives that
\begin{eqnarray*}
\!\!\!\!\!\!\!\!&&\mathbb{E}\left[\sup_{s\in[0,t]}\|X^\epsilon(s)-X^{\tilde{\epsilon}}(s)\|_{F_{1,2}^*}^2\right]
+\frac{\alpha_1}{2}\mathbb{E}\int_0^t|\Psi(s,X^\epsilon(s),\mu^\epsilon(s))-\Psi(s,X^{\tilde{\epsilon}}(s),\mu^{\tilde{\epsilon}}(s))|_2^2ds
\nonumber\\
\leq\!\!\!\!\!\!\!\!&&C_4\mathbb{E}\int_0^t\|X^\epsilon(s)-X^{\tilde{\epsilon}}(s)\|_{F_{1,2}^*}^2ds+C_2(\epsilon+\tilde{\epsilon})\mathbb{E}\int_0^t|\Psi(s,X^\epsilon(s),\mu^\epsilon(s))|_2^2+|\Psi(s,X^{\tilde{\epsilon}}(s),\mu^{\tilde{\epsilon}}(s))|_2^2ds
\nonumber\\
\!\!\!\!\!\!\!\!&&+8\mathbb{E}\Big(\int_0^t\|B (s,X^\epsilon(s),\mu^\epsilon(s))-B(s,X^{\tilde{\epsilon}}(s),\mu^{\tilde{\epsilon}}(s)\|_{L_2(U,H)}^2\|X^\epsilon(s)-X^{\tilde{\epsilon}}(s)\|_{F_{1,2}^*}^2ds\Big)^{\frac{1}{2}}
\nonumber\\
\leq\!\!\!\!\!\!\!\!&&C_5\mathbb{E}\int_0^t\|X^\epsilon(s)-X^{\tilde{\epsilon}}(s)\|_{F_{1,2}^*}^2ds+C_2(\epsilon+\tilde{\epsilon})\mathbb{E}\int_0^t|\Psi(s,X^\epsilon(s),\mu^\epsilon(s))|_2^2+|\Psi(s,X^{\tilde{\epsilon}}(s),\mu^{\tilde{\epsilon}}(s))|_2^2ds
\nonumber\\
\!\!\!\!\!\!\!\!&&+\frac{1}{2}\mathbb{E}\left[\sup_{s\in[0,t]}\|X^\epsilon(s)-X^{\tilde{\epsilon}}(s)\|_{F_{1,2}^*}^2\right],
\end{eqnarray*}
where the constants $C_4,C_5>0$ are independent of $\epsilon,\tilde{\epsilon}$.

Take $(\ref{6})$ into account we have
\begin{eqnarray}\label{28}
\!\!\!\!\!\!\!\!&&\mathbb{E}\left[\sup_{s\in[0,t]}\|X^\epsilon(s)-X^{\tilde{\epsilon}}(s)\|_{F_{1,2}^*}^2\right]
+\alpha_1\mathbb{E}\int_0^t|\Psi(s,X^\epsilon(s),\mu^\epsilon(s))-\Psi(s,X^{\tilde{\epsilon}}(s),\mu^{\tilde{\epsilon}}(s))|_2^2ds
\nonumber\\
\leq\!\!\!\!\!\!\!\!&& C(\epsilon+\tilde{\epsilon}).
\end{eqnarray}
Consequently, there is a continuous $(\mathcal{F}_t)_{t\geq 0}$-adapted process $X\in L^2(\Omega;C([0,T];F_{1,2}^*))$ such that $X^\epsilon\to X$ strongly in $L^2(\Omega;C([0,T];F_{1,2}^*))$ as $\epsilon\to 0$. The embedding $F_{1,2}^*\subset(L^2(\mu_M))^*$ densely implies that $X^\epsilon\to X$ weakly in $L^2(\Omega\times[0,T];L^2(\mu_M))$ as $\epsilon\to 0$. The Burkholder-Davis-Gundy's inequality yields that $\int_0^{\cdot}B(s,X^\epsilon(s),\mathcal{L}_{X^\epsilon(s)})dW(s)$ converges strongly to
$$\int_0^{\cdot}B(s,X(s),\mathcal{L}_{X(s)})dW(s)$$ in $L^2(\Omega;C([0,T];F_{1,2}^*))$ as $\epsilon\to 0$. Moreover, $$\int_0^\cdot\Psi(s,X^\epsilon(s),\mathcal{L}_{X^\epsilon(s)})ds$$
converges strongly to an element in $L^2(\Omega;C([0,T];F_{1,2}))$, owing to (\ref{28}), it follows that $\{\Psi(\cdot,X^\epsilon(\cdot),\mathcal{L}_{X^\epsilon(\cdot)})\}_{\epsilon\in(0,1)}$  strongly converges to an element noted by $Z(\cdot)$ in $L^2(\Omega\times[0,T];L^2(\mu_M))$.

For any $t\in[0,T]$, recall the equation
\begin{equation*}
X^\epsilon(t)+(\epsilon-L)\int_0^t\Psi(s,X^\epsilon(s),\mathcal{L}_{X^\epsilon(s)})ds=X(0)+\int_0^tB(s,X^\epsilon(s),\mathcal{L}_{X^\epsilon(s)})dW(s).
\end{equation*}
Taking $\epsilon\to 0$, it follows that for any $t\in[0,T]$,
\begin{equation*}
X(t)-L\int_0^t Z(s)ds=X(0)+\int_0^tB(s,X(s),\mathcal{L}_{X(s)})dW(s)~\text{holds in}~ (L^2(\mu_M))^*.
\end{equation*}

Repeating the same arguments as in the proof of Theorem \ref{t2}, we conclude $Z(\cdot)=\Psi(\cdot,X(\cdot),\mathcal{L}_{X(\cdot)})$, $dt\times\mathbb{P}$-a.s., and then it follows that $\int_0^\cdot\Psi(s,X(s),\mu(s))ds\in C([0,T];F_{1,2})$, $\mathbb{P}$-a.s.

The proof of the strong uniqueness of solutions to (\ref{e1}) is similar to the proof of uniqueness in Theorem \ref{t2}, here we omit the details. Since the strong solution is also a weak solution, it suffices to give the proof of weak uniqueness of solutions. Indeed,  the weak uniqueness of solutions can not be obtained directly since the classical Yamada-Watanabe theorem  is not directly applicable
in the distribution dependence case.

\textbf{(Weak uniqueness)} Given two weak solutions $(X(t),W(t))$ and $(\tilde{X}(t),\tilde{W}(t))$ with respect to the stochastic basis $(\Omega,\{\mathcal{F}_t\}_{t\geq 0},\mathbb{P})$ and $(\tilde{\Omega},\tilde{\{\mathcal{F}_t\}}_{t\geq 0},\tilde{\mathbb{P}})$, respectively, such that $\mathcal{L}_{X(0)}|_{\mathbb{P}}=\mathcal{L}_{\tilde{X}(0)}|_{\tilde{\mathbb{P}}}\in\mathcal{P}_2(H)$. Here we use $\mathcal{L}_{X(t)}|_{\mathbb{P}}$ to stress the distribution of $X(t)$ under probability measure $\mathbb{P}$. Denote $A:=L\Psi$, $X(t)$ solves Eq.~(\ref{e1}) and $\tilde{X}(t)$ solves the following
\begin{equation}\label{26}
d\tilde{X}(t)=A(t,\tilde{X}(t),\mathcal{L}_{\tilde{X}(t)}|_{\tilde{\mathbb{P}}})dt+B(t,\tilde{X}(t),\mathcal{L}_{{\tilde{X}(t)}} |_{\tilde{\mathbb{P}}})d\tilde{W}(t).
\end{equation}
Our aim is to verify $\mathcal{L}_{X(t)}|_{\mathbb{P}}=\mathcal{L}_{\tilde{X}(t)}|_{\tilde{\mathbb{P}}}$, $t\geq 0$. Let us denote $\mu(t)=\mathcal{L}_{X(t)}|_{\mathbb{P}}$, and consider $\bar{A}(t,x):=A(t,x,\mu)$ and $\bar{B}(t,x):=B(t,x,\mu)$, $x\in H$.

According to the conditions ({\textbf{A}}{\textbf{1}})-({\textbf{A}}{\textbf{4}}), the following SPDE
\begin{equation}\label{25}
d\bar{X}(t)=\bar{A}(t,\bar{X}(t))dt+\bar{B}(t,\bar{X}(t))d\tilde{W}(t),~\bar{X}(0)=\tilde{X}(0),,
\end{equation}
has a unique solution under $(\tilde{\Omega},\tilde{\{\mathcal{F}_t\}}_{t\geq 0},\tilde{\mathbb{P}})$. By making use of Yamada-Watanabe theorem, the weak uniqueness to Eq.~(\ref{25}) also holds. We note that
$$dX(t)=\bar{A}(t,X(t))dt+\bar{B}(t,X(t))dW(t),~\mathcal{L}_{X(0)}|_{\mathbb{P}}=\mathcal{L}_{\tilde{X}(0)}|_{\tilde{\mathbb{P}}},$$
the weak uniqueness of Eq.~(\ref{25}) gives that
\begin{equation}\label{27}
\mathcal{L}_{X(t)}|_{\mathbb{P}}=\mathcal{L}_{\bar{X}(t)}|_{\tilde{\mathbb{P}}}.
\end{equation}
Then it is obvious that Eq. (\ref{25}) reduces to
$$d\bar{X}(t)=A(t,\bar{X}(t),\mathcal{L}_{\bar{X}(t)}|_{\tilde{\mathbb{P}}})dt+B(t,\bar{X}(t),\mathcal{L}_{\bar{X}(t)}|_{\tilde{\mathbb{P}}})d\tilde{W}(t),~\bar{X}(0)=\tilde{X}(0).$$
By ({\textbf{A}}{\textbf{1}})-({\textbf{A}}{\textbf{4}}), (\ref{26}) has a unique solution and it follows that $\bar{X}=\tilde{X}$. Consequently, we conclude that $\mathcal{L}_{X(t)}|_{\mathbb{P}}=\mathcal{L}_{\tilde{X}(t)}|_{\tilde{\mathbb{P}}}$ by plugging $\bar{X}=\tilde{X}$ into (\ref{27}). The proof is complete.
\hspace{\fill}$\Box$

\vspace{1mm}
\noindent\textbf{Acknowledgements} {The authors would like to thank the referee  for helpful suggestions and comments.}

\end{document}